\documentclass[10pt]{amsart}
\usepackage[utf8]{inputenc}
\usepackage{amsfonts}
\usepackage{graphics}
\usepackage[all, cmtip]{xy}
\usepackage{amsthm,amsfonts,amssymb,amsmath,amsxtra}
\usepackage[all]{xy}
\usepackage{amsmath}
\usepackage{flexisym}
\usepackage{mathrsfs}
\usepackage{amsfonts}
\usepackage{amsmath}
\usepackage[colorlinks]{hyperref}
\hypersetup{
  citecolor   = blue 
 }

\newcommand{\rmM}{{\rm M}}
\newcommand{\tspan}{\text{span}}

\newcommand{\Res}{{\rm Res}}

\newcommand{\Tr}{{\rm Tr}}

\newcommand{\Gr}{{\rm Gr}}

\newcommand{\SO}{{\rm SO}}
\newcommand{\GU}{{\rm GU}}
\newcommand{\Sh}{{\rm Sh}}

\newcommand{\QQ}{\mathbb{Q}}

\newcommand{\calF}{\mathcal{F}}

\newcommand{\calM}{\mathcal{M}}

\newcommand{\calI}{\mathcal{I}}
\newcommand{\calJ}{\mathcal{J}}
\newcommand{\calU}{\mathcal{U}}

\newcommand{\bb }{\langle}
\newcommand{\pp}{\rangle}

\newcommand{\rightarroweq}{\stackrel{\sim}{\rightarrow}}

\usepackage{multicol}
\numberwithin{equation}{subsection}

\newtheorem{Theorem}{Theorem}[section]
\newtheorem{Remark}[Theorem]{Remark}
\newtheorem{Remarks}[Theorem]{Remarks}
\newtheorem{Lemma}[Theorem]{Lemma}
\newtheorem{Proposition}[Theorem]{Proposition}
\newtheorem{Corollary}[Theorem]{Corollary}

\newtheorem{Main Theorem}[Theorem]{Main Theorem}

\usepackage[colorlinks]{hyperref}
\makeatletter
\renewcommand*\env@matrix[1][*\c@MaxMatrixCols c]{%
  \hskip -\arraycolsep
  \let\@ifnextchar\new@ifnextchar
  \array{#1}}
\makeatother

\newif\ifgrading

\gradingfalse

\usepackage[colorlinks]{hyperref}
\hypersetup{linkcolor=black}
\usepackage{xcolor}
\usepackage{ wasysym }

\usepackage{tikz-cd}
\newcommand{\U}{{\mathcal U}}
\newcommand{\calG}{{\mathcal G}}
\newcommand{\calO}{{\mathcal O}}
\newcommand{\Mloc}{{\rm M}^{\rm loc}}
\newcommand{\Spec}{{\rm Spec \, } }
\newcommand{\wti}{\widetilde}
\newcommand{\und}{\underline}
\usepackage{xcolor}

\newcommand{\Addresses}{{
		\bigskip
		\footnotesize
		
		\textsc{Department of Mathematics, Universität Münster, Münster, 48149, Germany}\par\nopagebreak
		\textit{E-mail address:} \texttt{io.zachos@uni-muenster.de}\\

		\textsc{Department of Applied Mathematics,
University of Science and Technology Beijing, Beijing, 100083, China}\par\nopagebreak
		\textit{E-mail address:} \texttt{zhihaozhao@ustb.edu.cn}
		
}}	

\begin{document}

\title[Semi-Stable models for ramified Unitary Shimura Varieties]{Semi-stable and splitting models for unitary Shimura varieties over ramified
places. I.} 
	\date{}
	\author{I. Zachos and Z. Zhao}
	\maketitle

	\begin{abstract}
 	We consider Shimura varieties associated to a unitary group 
  of signature $(n-s,s)$ where $n$ is even. For these varieties, we construct smooth $p$-adic integral models for $s=1$ and regular $p$-adic integral models for $s=2$ and $s=3$ over odd primes $p$ which ramify in the imaginary quadratic field with level subgroup at $p$ given by the stabilizer of a $\pi$-modular lattice in the hermitian space. Our construction,  which has an explicit moduli-theoretic description, is given by an explicit resolution of a corresponding local model.
  
	\end{abstract}
	
	\tableofcontents

\section{Introduction}\label{Intro}

\subsection{} 

The main objective of this paper is to construct regular integral models for some Shimura varieties over places of bad reduction. 
Results of this type have already been obtained by  Rapoport-Zink \cite{RZbook}, Pappas \cite{P_ICM}, and He-Pappas-Rapoport \cite{HPR} in specific instances.
Here, we consider Shimura varieties associated to unitary similitude groups of signature $(r,s)$ over an imaginary quadratic field $K$. These Shimura varieties are moduli spaces of abelian varieties with polarization, endomorphisms and level structure (Shimura varieties of PEL type). Shimura varieties have canonical models over the “reflex” number field $E$ and in the cases we consider the reflex field is the field of rational numbers $\mathbb{Q}$ if $r = s$ and $E = K$ otherwise \cite[\S 3]{P}.

Constructing such well-behaved integral models is an interesting and challenging problem, with several applications to number theory.
The behavior of these depends very much on the “level subgroup”. Here, the level subgroup is the stabilizer of a $\pi$-modular lattice in the hermitian space. (
This is a lattice which is equal to its dual times a uniformizer). For this level subgroup, and the signatures $(n-s,s)$ with $n$ even and $s \leq 3$, we show that there is a good (semi-stable) model.
 
 By using work of Rapoport-Zink \cite{RZbook} and Pappas \cite{P} we first construct $p$-adic integral models, which have simple and explicit moduli descriptions but are not always flat. These models are \'etale locally around each point isomorphic to certain simpler schemes the \textit{naive local models}. Inspired by the work of Pappas-Rapoport \cite{PR2} and Krämer \cite{Kr}, we consider a variation of the above moduli problem 
 where we add in the moduli problem an additional subspace in the deRham filtration $ {\rm Fil}^0 (A) \subset H_{dR}^1(A)$ of the universal abelian variety $A$, which satisfies certain conditions. This is essentially an instance of the notion of a ``linear modification" introduced in \cite{P}. Then for the signatures $(n-s,s)$, with $s \leq 3$, we calculate the flat closure of these models and we show that they are smooth when $s=1$ and have semi-stable reduction when $s=2$ or $s=3$, i.e. they are regular and the irreducible components of the special fiber are smooth divisors crossing normally. Moreover, we want to mention that we obtain a moduli description of this flat closure.
 We anticipate that our construction will have applications to the study of arithmetic intersections of special cycles and Kudla’s program. (See for example, \cite{Zhang}, \cite{BHKR} and \cite{HLSY}, for some works in this direction.)

\subsection{} Let us introduce some notation first. Recall that $K$ is an imaginary quadratic field and we fix an embedding $\varepsilon: K\rightarrow \mathbb{C}$. Let $W$ be a $n$-dimensional $K$-vector space, equipped with a non-degenerate hermitian form $\phi$. Consider the group $G = \GU_n$ of unitary similitudes for $(W,\phi)$ of dimension $n\geq 3$ over $K$. We fix a conjugacy class of homomorphisms $h: \Res_{\mathbb{C}/\mathbb{R}}\mathbb{G}_{m,\mathbb{C}}\rightarrow \GU_n$ corresponding to a Shimura datum $(G,X_h)$ of signature $(r,s)$. Set $X = X_h$. The pair $(G,X)$ gives rise to a Shimura variety $Sh(G,X)$ over the reflex field $E$ (see \cite[\S 1.1]{PR} for more details). Let $p$ be an odd prime number which ramifies in $K$. Set $K_1=K\otimes_{\QQ} \QQ_p$ with a uniformizer $\pi$, and $V=W\otimes_{\QQ} \QQ_p$. 
We assume that the hermitian form $\phi$ is split on $V$, i.e. there is a basis $e_1, \dots, e_n$ such that $\phi(e_i, e_{n+1-j})=\delta_{ij}$ for $1\leq i, j\leq n$. We denote by 
\[
\Lambda_i = \text{span}_{O_{K_1}} \{\pi^{-1}e_1, \dots, \pi^{-1}e_i, e_{i+1}, \dots, e_n\}
\]
the standard lattices in $V$. We can complete this into a self dual lattice chain by setting $\Lambda_{i+kn}:=\pi^{-k}\Lambda_i$ (see \S \ref{Prel.1}).

By \cite{PR}, there are 3 different types of the special maximal parahoric subgroups of $\GU_n$ depending on the parity of $n$. More precisely, 
consider the stabilizer subgroup
\[
P_I:=\{g\in\GU_n\mid g\Lambda_i=\Lambda_i , \quad \forall i \in I\},
\]
where $I$ is a nonempty subset of $\{0,\dots, \lfloor n/2\rfloor\}$. When $n=2m+1$ is odd, the special maximal parahoric subgroups are conjugate to the stabilizer subgroup $P_{\{0\}}$ or $P_{\{m\}}$. When $n=2m$ is even, the special maximal parahoric subgroups are conjugate to the stabilizer subgroup $P_{\{m\}}$. In this article, we consider the special maximal parahoric subgroup $P_{\{m\}}$ when $n=2m$ is even. We intend to take up the case $n=2m+1, I=\{m\}$ in a subsequent work.

To explain our results, we assume $n=2m$, and $(r,s)=(n-1,1)$ or $(n-2,2)$ or $(n-3,3)$. Denote by $P_{\{m\}}$ the stabilizer of $\Lambda_m$ in $G(\mathbb{Q}_p)$. We let $\mathcal{L}$ be the
self-dual multichain consisting of lattices $\{\Lambda_j\}_{j\in n\mathbb{Z} \pm m}$. Here $\mathcal{G} = \underline{{\rm Aut}}(\mathcal{L})$ is the (smooth) group scheme over $\mathbb{Z}_p$ with $P_{\{m\}} = \mathcal{G}(\mathbb{Z}_p)$ the subgroup of $G(\mathbb{Q}_p)$ fixing the lattice chain $\mathcal{L}$.   

Choose also a sufficiently small compact open subgroup $K^p$ of the prime-to-$p$ finite adelic points $G({\mathbb A}_{f}^p)$ of $G$ and set $\mathbf{K}=K^pP_{\{m\}}$. The Shimura variety  ${\rm Sh}_{\mathbf{K}}(G, X)$ with complex points
 \[
 {\rm Sh}_{\mathbf{K}}(G, X)(\mathbb{C})=G(\mathbb{Q})\backslash X\times G({\mathbb A}_{f})/\mathbf{K}
 \]
is of PEL type and has a canonical model over the reflex field $E$. We set $\mathcal{O} = O_{E_v}$ where $v$ the unique prime ideal of $E$ above $(p)$.

We consider the moduli functor $\mathcal{A}^{\rm naive}_{\mathbf{K}}$ over $\Spec \mathcal{O} $ given in \cite[Definition 6.9]{RZbook}:\\
A point of $\mathcal{A}^{\rm naive}_{\mathbf{K}}$ with values in the 
$\mathcal{O}  $-scheme $S$ is the isomorphism class of the following set of data $(A,\iota,\bar{\lambda}, \bar{\eta})$:
\begin{enumerate}
   \item An object $(A,\iota)$, where $A$ is an abelian scheme with relative dimension $n$ over $S$ (terminology
of \cite{RZbook}), compatibly endowed with an
action of $\calO$: 
\[ \iota: \calO \rightarrow \text{End} \,A \otimes \mathbb{Z}_p.\]
    \item A $\mathbb{Q}$-homogeneous principal polarization $\bar{\lambda}$ of the $\mathcal{L}$-set $A$.
    \item A $K^p$-level structure
    \[
\bar{\eta} : H_1 (A, {\mathbb A}_{f}^p) \simeq W \otimes  {\mathbb A}_{f}^p \, \text{ mod} \, K^p
    \]
which respects the bilinear forms on both sides up to a constant in $({\mathbb A}_{f}^p)^{\times}$ (see loc. cit. for
details). The set $A$ should satisfy the determinant condition (i) of loc. cit. which depends on $(r,s)$.
\end{enumerate}

For the definitions of the terms employed here we refer to loc.cit., 6.3–6.8 and \cite[\S 3]{P}. The functor $\mathcal{A}^{\rm naive}_{\mathbf{K}}$ is representable by a quasi-projective scheme over $\mathcal{O}$. Since the Hasse principle is satisfied for the unitary group, we can see as in loc. cit. that there is a natural isomorphism
\[
\mathcal{A}^{\rm naive}_{\mathbf{K}} \otimes_{\calO} E_v = {\rm Sh}_{\mathbf{K}}(G, X)\otimes_{E} E_v.
\]

The moduli scheme $\mathcal{A}^{\rm naive}_{\mathbf{K}}$ is connected to the naive local model ${\rm M}^{\rm naive}$ (see \S \ref{LocalModelVariants} for the explicit definition). As is explained in \cite{RZbook} and \cite{P} the naive local model  is connected to the moduli scheme $\mathcal{A}^{\rm naive}_{\mathbf{K}}$ via the local model diagram 
\[
\begin{tikzcd}
&\wti{\mathcal{A}}^{\rm naive}_{\mathbf{K}}(G, X)\arrow[dl, "\psi_1"']\arrow[dr, "\psi_2"]  & \\
\mathcal{A}^{\rm naive}_{\mathbf{K}}  && {\rm M}^{\rm naive}
\end{tikzcd}
\]
where the morphism $\psi_1$ is a $\mathcal{G}$-torsor and $\psi_2$ is a smooth and $\mathcal{G}$-equivariant morphism. In particular, since $\calG$ is smooth, the above imply that $\mathcal{A}^{\rm naive}_{\mathbf{K}} $ is \'etale locally isomorphic to ${\rm M}^{\rm naive}$. From \cite[Remark 2.6.10]{PRS}, we have that the naive local model is never flat  and by the above the
same is true for $\mathcal{A}^{\rm naive}_{\mathbf{K}} $. Denote by $ \mathcal{A}^{\rm flat}_{\mathbf{K}} $ the flat closure of ${\rm Sh}_{\mathbf{K}}(G, X)\otimes_{E} E_v$ in $ \mathcal{A}^{\rm naive}_{\mathbf{K}}$. As in \cite{PR}, there is a relatively representable smooth morphism of relative dimension ${\rm dim} (G)$,
\[
\mathcal{A}^{\rm flat}_{\mathbf{K}} \to [\mathcal{G} \backslash \Mloc]
\]
where the local model $\Mloc $ is defined as the flat closure of $ {\rm M}^{\rm naive} \otimes_{\mathcal{O}} E_v$ in ${\rm M}^{\rm naive}$. This of course implies that $\mathcal{A}^{\rm flat}_{\mathbf{K}}$ is \'etale locally isomorphic to the local model $\Mloc$.

We now consider a variation of the moduli of abelian schemes $\mathcal{A}_{\mathbf{K}}$ where we add in the moduli problem an additional subspace in the Hodge filtration $ {\rm Fil}^0 (A) \subset H_{dR}^1(A)$ of the universal abelian variety $A$ (see \cite[\S 6.3]{H} for more details) with certain conditions to imitate the definition of the naive splitting model $\mathcal{M}$ (we refer to \S \ref{LocalModelVariants} for the explicit definition). 
There is a projective forgetful 
morphism
\[
\tau_1 :   \mathcal{A}_{\mathbf{K}} \longrightarrow \mathcal{A}^{\rm naive}_{\mathbf{K}}\otimes_{\mathcal{O}} O_{K_1}\]
which induces an isomorphism over the generic fibers (see \S \ref{Shimura}). Moreover, $\mathcal{A}_{\mathbf{K}}$ has the same \'etale local structure as $\mathcal{M}$; it is a ``linear modification" of $\mathcal{A}^{\rm naive}_{\mathbf{K}}\otimes_{\mathcal{O}} O_{K_1}$ in the sense of \cite[\S 2]{P} (see also \cite[\S 15]{PR2}). Note, that there is also a corresponding projective forgetful 
morphism
\[
\tau :   \mathcal{M} \longrightarrow  {\rm M}^{\rm naive} \otimes_{\mathcal{O}} O_{K_1}
\]
which induces an isomorphism over the generic fibers (see \S \ref{LocalModelVariants}).
In \S \ref{spl}, we show that $\calM$ is not flat for any signature $(r,s)$ and the same is true for $ \mathcal{A}_{\mathbf{K}}$. In fact, $\calM$ is not topologically flat since there exists a point which cannot lift to characteristic zero (see 
Proposition \ref{naive.notflat}). This type of phenomenon is similar to what already has been observed for
this level subgroup for the wedge local models in \cite[Remark 5.3]{PR} (see \S \ref{LocalModelVariants} for the definition of wedge local models). Consider the closed subscheme $\mathcal{A}^{\rm spl}_{\mathbf{K}} \subset \mathcal{A}_{\mathbf{K}} $ defined as $\mathcal{A}^{\rm spl}_{\mathbf{K}} :=\tau_1^{-1} (\mathcal{A}^{\rm flat}_{\mathbf{K}}) $. It is a linear modification of $\mathcal{A}^{\rm naive}_{\mathbf{K}}\otimes_{\mathcal{O}} O_{K_1}$.


\subsection{} One of the main results of the present paper is the following
theorem.

 \begin{Theorem}
Assume that $ (r,s) = (n-1,1)$ or $(n-2,2)$ or $(n-3,3)$. For every $K^p$ as above, there is a
 scheme $\mathcal{A}^{\rm spl}_{\mathbf{K}}$, flat over $\Spec(O_{K_1})$, 
 with
 \[
\mathcal{A}^{\rm spl}_{\mathbf{K}}\otimes_{O_{K_1}} K_1={\rm Sh}_{\mathbf{K}}(G, X)\otimes_{E} K_1,
 \]
 and which supports a local model diagram
  \begin{equation}\label{LMdiagramRegIntro}
\begin{tikzcd}
&\wti{\mathcal{A}}^{\rm spl}_{\mathbf{K}}(G, X)\arrow[dl, "\pi^{\rm reg}_K"']\arrow[dr, "q^{\rm reg}_K"]  & \\
\mathcal{A}^{\rm spl}_{\mathbf{K}}  &&  {\rm M}^{\rm spl}
\end{tikzcd}
\end{equation}
such that:
\begin{itemize}
\item[a)] $\pi^{\rm reg}_{\mathbf{K}}$ is a $\mathcal{G}$-torsor for the parahoric group scheme $\calG$ that corresponds to $P_{\{m\}}$.

\item[b)] $q^{\rm reg}_{\mathbf{K}}$ is smooth and $\calG$-equivariant.

\item[c)] When $(r,s) = (n-1,1)$, $\mathcal{A}^{\rm spl}_{\mathbf{K}}$ is a smooth scheme. 
\item[c')] When $(r,s) = (n-2,2)$ or $(r,s) =(n-3,3)$, $\mathcal{A}^{\rm spl}_{\mathbf{K}}$ has semi-stable reduction. In particular, $\mathcal{A}^{\rm spl}_{\mathbf{K}}$ is regular and has special fiber which is a reduced divisor with
normal crossings.
\end{itemize}
 \end{Theorem}
 
In the above, the splitting model ${\rm M}^{\rm spl}$ is defined as $  {\rm M}^{\rm spl}:=\tau^{-1}({\rm M}^{\rm loc})$. Since every point of $\mathcal{A}^{\rm spl}_{\mathbf{K}} $ has an \'etale neighborhood which is also \'etale over ${\rm M}^{\rm spl}$, it is enough to show that ${\rm M}^{\rm spl}$ has the above nice properties. To show this, we explicitly calculate an affine chart $\calU$ of ${\rm M}^{\rm spl}$ in a neighbourhood of $  \tau^{-1}(*)$ where $*$ is a point from the unique closed $\mathcal{G}$-orbit supported in the special fiber of ${\rm M}^{\rm loc}$ (see \cite[\S 4]{P} for more details). Note that the unique closed $\mathcal{G}$-orbit depends on the parity of $s$ (see \S \ref{LocalModelVariants}). We treat these two cases, i.e. when $s$ is even and odd, separately in \S \ref{AffineChart1} and \S \ref{AffineChart2}.

Moreover, we give a moduli-theoretic description of ${\rm M}^{\rm spl}$ and so by the above for $\mathcal{A}^{\rm spl}_{\mathbf{K}}$. Inspired by the work of Pappas-Rapoport \cite{PR}, we define the closed subscheme $ \mathcal{M}^{\rm spin} \subset \mathcal{M}$ by adding the ``spin condition'' (see \S \ref{mod_des}) and we show

\begin{Theorem}\label{SpinIntro}
For the signatures $(n-s,s)$ with $s\leq 3$, we have $ \mathcal{M}^{\rm spin}={\rm M}^{\rm spl}$. 
\end{Theorem}

Here, we note that at the level of local models, a moduli description of $\Mloc$ for the signature $(n-1,1)$ can be obtained by adding the `` wedge and spin condition" to the moduli problem of ${\rm M}^{\rm naive}$ (see \cite{RSZ17}). For higher signatures, it has been conjectured that we can get the moduli description of $\Mloc$ by adding the ``wedge and spin conditions" to ${\rm M}^{\rm naive}$; see the survey paper \cite{PRS} for more details.

Before we move on to the description of $\calU$, we want to give a brief historical outline about the splitting models. 
Splitting models were first described for associated Shimura varieties of type $A$ and type $C$ by Pappas-Rapoport \cite{PR2}. For unitary groups, Kr\"{a}mer \cite{Kr} first showed that the splitting model has semi-stable reduction 
for the signature $(r,s)=(n-1,1)$ when $n=2m+1, I=\{0\}$. When the signature is $(r,s)=(n-2,2)$ 
the first author \cite{Zac1} showed that  the corresponding splitting model $\mathcal{M}$ does not have semi-stable reduction. Then, he resolved the singularities by giving an explicit blow-up of $\mathcal{M}$ along a smooth divisor. For more information on the geometry of the special fiber of $\mathcal{M}$ for general signature $(r,s)$ we refer the reader to the very recent work of Bijakowski and Hernandez \cite{BH}. Lastly, we want to mention the work of the second author where he considered the splitting model for triality groups in \cite{Zhao1} (see also \cite{Zhao2}). 

For the rest of this subsection, we fix $n=2m$ and $I=\{m\}$. 
Define
    \[
    J_n = \left[\ 
\begin{matrix}[c|c]
0_m & -H_m  \\ \hline
H_m & 0_m 
\end{matrix}\ \right]
    \]
where $H_m$ is the unit antidiagonal matrix (of size $m$). For a general signature $(r,s)$, we first obtain:

\begin{Theorem}

(1) When $s$ is even, there is an affine chart $\mathcal{U} \subset \calM$ which is isomorphic to $\Spec O_{K_1}[X,Y]/I $ where 
\[
I = \left(Y'-I_s, \,\, X+X^t, \,\, X\cdot  Y^t \cdot J_n \cdot Y  +2 \pi I_s\right)
\]
for some choice of $Y'$. Here $X$, $Y$ are matrices of sizes $s\times s$ and $n\times s$ respectively with indeterminates as entries and  
$Y'$ is a 
submatrix of $Y$ of size $  s \times s$, i.e. it is composed of $s$ rows from $Y$ along with the corresponding $s$ columns.

(2a) When $s=1$, there is an affine chart $\calU \subset \calM$  which is isomorphic to $\mathbb{A}_{O_{K_1}}^{n-1}$.

(2b) When $s$ is odd and $s\geq 3$, there is an affine chart $\calU \subset \calM$  which is isomorphic to $\mathbb{A}_{O_{K_1}}^{s-1}\times \Spec O_{K_1}[X, Y]/I$
where
\[
I=(Y'-[\mathbf{0}\mid I_{s-1}], \,\,  X+X^t, \,\, X\cdot  Y^t \cdot J_{n-2} \cdot Y +2\pi I_{s-1})
\]
for some choice of $Y'$. Here $X$, $Y$ are matrices of sizes $(s-1)\times (s-1)$ and $(n-2)\times s$ respectively with indeterminates as entries and 
$Y'$ is a submatrix of $Y$ of size $ (s-1)\times s$, i.e. it is composed of $(s-1)$ rows from $Y$ along with the corresponding $s$ columns.
\end{Theorem}

By using the explicit description of $\calU$ above, we prove that $\mathcal{G}$-translates of $\mathcal{U}$ cover ${\rm M}^{\rm spl}$ and we deduce:


  
\begin{Theorem}
\begin{itemize}
\item[a)] When $s=1$, ${\rm M}^{\rm spl}$ is a smooth scheme. 

\item[b)] When $s=2$ or $s=3$, ${\rm M}^{\rm spl}$ has semi-stable reduction. In particular, ${\rm M}^{\rm spl}$ 
is regular and has special fiber a reduced divisor with two smooth irreducible components intersecting transversely.
\end{itemize}
\end{Theorem}

When $s=2$ or $s=3$, we show, by using Theorem \ref{SpinIntro}, that one of the two components of the special fiber of ${\rm M}^{\rm spl}$ maps birationally to the special fiber of $\Mloc$ while the other irreducible component is the inverse image of the unique closed $\calG$-orbit. 

Moreover, for $(r,s)=(n-1,1) $, we deduce that $  {\rm M}^{\rm spl}$ is equal to the local model $\Mloc \otimes_{\mathcal{O}} O_{K_1}$ (see Remark \ref{loc_splt}). 
In \cite[\S 5.3]{PR}, the authors showed that $\Mloc$ is smooth. This is a case of ``exotic" good reduction. We also refer the reader to the result of \cite{HR} where the authors give an alternative explanation for the smoothness of $\Mloc$ by identifying its special fiber with a Schubert variety attached to a minuscule
cocharacter in the twisted affine Grassmannian corresponding to $P_{\{m\}}$.


Lastly, we want to mention that we can apply these results to obtain regular
(formal) models of the corresponding Rapoport-Zink spaces.

Let us now outline the contents of the paper: In \S \ref{Prelim}, we recall the parahoric subgroups of the unitary similitude group when $n=2m$. In \S \ref{LocalModelVariants}, we recall the definition of certain variants of local models for ramified unitary groups. The explicit equations of $\calU$ are given in \S \ref{Two affine charts}. In \S \ref{Naive.spl}, we show that the naive splitting model $\calM$ is not flat, and prove that $\calG$-translates of $\calU$ cover ${\rm M}^{\rm spl}$ for the signature $(r,s)$ ($s\leq 3$) in \S \ref{Evensemistable} and \S \ref{Odd_s} respectively. 
In \S \ref{Shimura}, we use the above results to construct smooth integral models for the signature $(r,s)=(n-1,1)$, and integral models with semi-stable reduction for $(r,s)=(n-2,2)$ and $(r,s)=(n-3,3)$. In \S \ref{mod_des}, we give a moduli-theoretic description of ${\rm M}^{\rm spl}$ for $s\leq 3$ by considering the spin condition.

\section{Preliminaries}\label{Prelim}
\subsection{Pairings and Standard Lattices}\label{Prel.1}
We use the notation of \cite{PR}. Let $F_0$ be a complete discretely valued field with ring of integers $O_{F_0}$, perfect residue field $k$ of characteristic $\neq 2$, and uniformizer $\pi_0$. Let $F/F_0$ be a ramified quadratic extension and $\pi \in F$ a uniformizer with $\pi^2 = \pi_0$.  Let $V$ be a $F$-vector space of dimension $n =2m> 3$ and let 
\[
\phi: V \times V \rightarrow F
\]
be an $F/F_0$-hermitian form. We assume that $\phi$ is split. This means that there exists a basis $e_1, \dots , e_n$ of $V$ such that 
\[
\phi(e_i,e_{n+1-j}) = \delta_{i,j} \quad \text{for  all} \quad i,j = 1, \dots, n.
\]
We attach to $\phi$ the respective alternating and symmetric $F_0$-bilinear forms $V \times V \rightarrow F_0$
\[
\langle x, y \rangle = \frac{1}{2}\text{Tr}_{F /F_0} (\pi^{-1}\phi(x, y)) \quad \text{and} \quad ( x, y ) =  \frac{1}{2}\text{Tr}_{F /F_0} (\phi(x, y)).
\]
For any $O_F$-lattice $\Lambda$ in $V$, we denote by
\[ 
\hat{\Lambda} = \{v \in V | \langle v, \Lambda \rangle \subset O_{F_0} \}
\]
the dual lattice with respect to the alternating form and by
\[
\hat{\Lambda}^s = \{v \in V | ( v, \Lambda ) \subset O_{F_0} \}
\]
the dual lattice with respect to the symmetric form. We have $ \hat{\Lambda}^s = \pi^{-1}\hat{\Lambda}.$ Both $\hat{\Lambda}$
and $\hat{\Lambda}^s$ are $O_F$-lattices in V, and the forms $\langle \, , \, \rangle $ and $ (\, , \,) $
induce perfect $O_{F_0}$-bilinear pairings
\begin{equation}\label{perfectpairing}
    \Lambda \times \hat{\Lambda} \xrightarrow{\langle \, , \, \rangle } O_{F_0}, \quad \Lambda \times \hat{\Lambda}^s \xrightarrow{ (\, , \,)} O_{F_0}
\end{equation}
for all $\Lambda$. Also, the uniformizing element $\pi$ induces a $O_{F_0}$- linear mapping on $\Lambda$ which we denote by $t$.

For $i= 0, \dots, n- 1$, we define the standard lattices
\[
\Lambda_i = \text{span}_{O_F} \{\pi^{-1}e_1, \dots, \pi^{-1}e_i, e_{i+1}, \dots, e_n\}.
\]
We consider nonempty subsets $I \subset \{0, \dots, m\}$ with the property 
\begin{equation}\label{Iproperty}
m-1 \in I \Longrightarrow m \in I
\end{equation}
(see \cite[\S 1.2.3(b)]{PR} for more details). We complete the $\Lambda_i$ with $i \in I$ to a self-dual periodic lattice chain by first including the duals 
$ \Lambda_{n-i}:=\hat{\Lambda}^s_i $ for $i \neq 0$ and then all the $\pi$-multiples: For $j \in \mathbb{Z}$ of the form $j = k \cdot n \pm i$ with $i \in I$ we put $\Lambda_{j} = \pi^{-k}\Lambda_i$. Then $ \{\Lambda_j\}_j$ form a periodic lattice chain $\Lambda_I$ (with $\pi \Lambda_j =  \Lambda_{j-n}$) which satisfies $ \hat{\Lambda}_j = \Lambda_{-j}$. We denote by  $\mathcal{L}$ such a self-dual multichain. Observe that the lattice $\Lambda_0$ is self-dual for the alternating form $\langle\,,\,\rangle$ and $\Lambda_m$ is  self-dual for the symmetric form $(\,,\,)$.  


\subsection{Unitary Similitude Group and Parahoric Subgroups}\label{ParahoricSbgrs}
We consider the unitary similitude group 
\[
G:=\GU(V, \phi) = \{g \in GL_F (V ) \,\, | \,\, \phi(gx, gy) = c(g)\phi(x, y), \quad  c(g) \in F^{\times}_0 \}
\]
and we choose a partition $n = r + s$; we refer to the pair $(r, s)$ as the signature. By replacing $\phi$ by $-\phi$ if needed, we can make sure that $s\leq r$ and so we assume that $s \leq r$ (see \cite[\S 1.1]{PR} for more details). Identifying $G \otimes F \simeq GL_{n,F} \times \mathbb{G}_{m,F}$, we define the cocharacter $ \mu_{r,s}$ as $ (1^{(s)}, 0^{(r)}, 1)$ of $D \times \mathbb{G}_m$, where $D$ is the standard maximal torus of diagonal matrices in $GL_n$; for more details we refer the reader to \cite{Sm2}. We denote by $E$ the reflex field of $\{ \mu_{r,s}\}$; then $E = F_0$ if $r = s$ and $E = F$ otherwise (see \cite[\S 1.1]{PR}). We set $O := O_E$. 

We next recall the description of the parahoric subgroups of $G$ from \cite[\S 1.2]{PR}, which actually follows from the results on parahoric subgroups of ${\rm SU}(V,\phi)$ in \cite[\S 1.2]{PR3}. 
Recall that $n=2m$ is even and $I$ is a non-empty subset of $\{0,\dots,m\}$ satisfying (\ref{Iproperty}). 
Consider the subgroup
\[
P_I:=\{g\in G\mid g\Lambda_i=\Lambda_i , \quad \forall i \in I\}.
\]
The subgroup $P_I$ is not a parahoric subgroup in
general since it may contain elements with nontrivial Kottwitz invariant. Consider the kernel of the Kottwitz homomorphism:
\[
P_I^0:=\{g\in P_I\mid \kappa(g)=1\}
\]
where $\kappa_G : G(F_0) \twoheadrightarrow \mathbb{Z} \oplus (\mathbb{Z}/2\mathbb{Z})$ (see also \cite[\S 3]{Sm2} for more details). We have the following (see \cite[\S 1.2.3(b)]{PR}):

\begin{Proposition}
The subgroup $P_I^0$ is a parahoric subgroup and every parahoric subgroup of $G$ is conjugate to $P_I^0$ for a
unique nonempty $I \subset \{0,\dots,m\}$ satisfying (\ref{Iproperty}). For such $I$, we have $P_I^0 = P_I$ exactly when $m \in I$. The special maximal parahoric subgroups are exactly those
conjugate to $P^0_{\{m\}} = P_{\{m\}}$.\qed   
\end{Proposition}
In this paper, we will focus on the special maximal parahoric subgroup when $n$ is even. So, we fix  $n=2m$, $I=\{m\}$ and we let $\mathcal{L}$ be the multichain defined in \ref{Prel.1} for this choice of $I$. Denote by $\calG$ the (smooth) group scheme of automorphisms 
of the polarized chain $\mathcal{L}$ over $O_{F_0}$; then $\calG$ is the parahoric group model of $G$ in the sense of Bruhat-Tits \cite{BT84} with $\calG(O_{F_0}) =P_{\{m\}}$. It has connected fibers (see  \cite[\S 1.5]{PR}).


\section{Rapoport-Zink Local Models and Variants}\label{LocalModelVariants}
We briefly recall the definition of certain variants of local models for ramified unitary groups that correspond to the local model triples $(G, \mu_{r,s}, P_{\{m\}} )$. 

Let ${\rm M}^{\rm naive}$ be the functor which associates to each scheme $S$
over $\Spec O$ the set of subsheaves $\mathcal{F}$ of $O \otimes \mathcal{O}_S $-modules of $ \Lambda_m \otimes \mathcal{O}_S  $
such that
\begin{enumerate}
    \item $\mathcal{F}$ as an $\mathcal{O}_S $-module is Zariski locally on $S$ a direct summand of rank $n$;
    \item $\mathcal{F}$ is totally isotropic for $( \,
, \, ) \otimes \mathcal{O}_S$;
    \item (Kottwitz condition) $\text{char}_{t |  \mathcal{F} } (X)= (X + \pi)^r(X - \pi)^s $.
\end{enumerate}

\begin{Remarks}{\rm
 
 In \cite[\S 1.5.1]{PR}, the authors define the naive local model $\rmM^{\rm naive}$ that sends each $O$-scheme $S$ to the families of $O\otimes \calO_S$-modules $(\calF_i\subset \Lambda_i\otimes\calO_S)_{i \in n\mathbb{Z}\pm I}$ that satisfy the conditions (a)-(d) of loc. cit. We want to explain how we get the isotropic condition (2) from the condition (c) of loc. cit. When $I=\{m\}$, the complete lattice chain gives:
 \[
 \cdots\rightarrow \Lambda_{-m}\otimes\calO_S\rightarrow \Lambda_{m}\otimes\calO_S\rightarrow \cdots
 \]
where the isomorphism $t:\Lambda_{m}\otimes\calO_S\rightarrow \Lambda_{-m}\otimes\calO_S$ induces an isomorphism $\calF_{m}$ with $\calF_{-m}$. Hence, $\calF_{-m}$ is determined by $\calF_m$ and we have $\calF_{-m}=t\calF_{m}$. The perfect bilinear pairing 
\[
 \bb \ , \ \pp\otimes\calO_S: (\Lambda_{-m}\otimes\calO_S)\times (\Lambda_{m}\otimes\calO_S)\rightarrow \calO_S
 \]
induced by (\ref{perfectpairing}) identifies $\calF_m^\bot \subset \Lambda_{-m}\otimes\calO_S$ with $\calF_{-m}$ where $\calF_m^\bot$ is the orthogonal complement of $\mathcal{F}_m$ under the above perfect
pairing. Thus, $\bb \calF_{-m},\calF_m\pp=0$ (condition (c) of loc. cit.) is equivalent to $(\calF_m,\calF_m)=0$ since
 \begin{flalign*}
 	\bb \calF_{-m},\calF_m\pp &=\bb t\calF_{m},\calF_m\pp=\frac12\Tr_{F/F_0}(\pi^{-1}\phi(t\calF_{m},\calF_m))\\
 	&=\frac12\Tr_{F/F_0}(\phi(\calF_{m},\calF_m))=(\calF_m,\calF_m).
 \end{flalign*}
 }\end{Remarks}
 
The functor ${\rm M}^{\rm naive}$ is represented by a closed subscheme, which we again
denote ${\rm M}^{\rm naive}$, of $\Gr(n, 2n) \otimes \Spec O$; hence ${\rm M}^{\rm naive}$ is a projective $ O$-scheme. (Here we denote by $\Gr(n, d)$ the Grassmannian scheme parameterizing locally direct summands of rank $n$ of a free module of rank $d$.) The scheme ${\rm M}^{\rm naive}$ is the \textit{naive local model} of  Rapoport-Zink \cite{RZbook}. Also, ${\rm M}^{\rm naive}$ supports an action of $\mathcal{G}$. 
\begin{Proposition}\label{notflat}
We have 
\[
{\rm M}^{\rm naive} \otimes_{ O} E \cong \Gr(s,n)\otimes E.
\]
In particular, the generic fiber of ${\rm M}^{\rm naive}$ is smooth and geometrically irreducible of dimension $rs$.
\end{Proposition}
\begin{proof}
See \cite[\S 1.5.3]{PR}. 
\end{proof}
Next, we consider the closed subscheme ${\rm M}^{\wedge}\subset {\rm M}^{\rm naive}$ by imposing the following additional condition: 
\[
 \wedge^{r+1} (t-\pi | \mathcal{F}) = (0) \quad \text{and} \quad  \wedge^{s+1} (t+\pi | \mathcal{F}) = (0). 
\]
More precisely, ${\rm M}^{\wedge}$ is the closed subscheme of ${\rm M}^{\rm naive}$ that classifies points given by $\mathcal{F}$ which satisfy the wedge condition. The scheme ${\rm M}^{\wedge}$ supports an action of $\mathcal{G}$ and the immersion $i : {\rm M}^{\wedge} \rightarrow
{\rm M}^{\rm naive}$ is $\mathcal{G}$-equivariant. The scheme ${\rm M}^{\wedge}$ is called the \textit{wedge local model}.
\begin{Proposition}
The generic fibers of ${\rm M}^{\rm naive} $ and ${\rm M}^{\wedge} $ coincide, in particular
the generic fiber of ${\rm M}^{\wedge}$ is a smooth, geometrically irreducible variety of
dimension rs.
\end{Proposition}
\begin{proof}
See  \cite[\S 1.5.6]{PR}. 
\end{proof}

For the special maximal parahoric subgroup $P_{\{0\}}$ and signature $(n-1,1)$, Pappas proved that the wedge local model ${\rm M}^{\wedge}$ is flat \cite{P}. But for more general parahoric subgroups, the wedge condition turns out to be insufficient (see \cite[Remark 5.3, 7.4]{PR}). There is a further variant: let $\Mloc$ be the scheme theoretic closure of the generic fiber $ {\rm M}^{\rm naive} \otimes_{O} E$ in ${\rm M}^{\rm naive}$. The scheme $\Mloc$ is called the \textit{local model}. We have closed immersions of projective schemes
\[
\Mloc \subset {\rm M}^{\wedge}\subset {\rm M}^{\rm naive}
\]
which are equalities on generic fibers (see \cite[\S 1.5]{PR} for more details).
\begin{Proposition}
a) For any signature $(r, s)$, the special fiber of $\Mloc$ is integral, normal, Cohen-Macaulay and has only rational singularities.

b) For $(r, s)= (n-1,1)$, $\Mloc$ is smooth. 
\end{Proposition}
\begin{proof}
See \cite[\S 5]{PR} and \cite{HR22}. 
\end{proof}
Except in the case (b) above (up to switching $(r, s)$ and $(s,r)$), the local models $\Mloc$ are
never smooth; see  \cite[Remark 2.6.8]{PRS} for more details. As in \cite[\S 2.4.2, \S 5.5]{PR},  there is a unique closed $\calG$-orbit in the special fiber of $\Mloc$. When $s$ is even the closed orbit is the $ k$-valued point 
\[
\mathcal{F} =\text{span}_{O_{F_0}} \{e_1, \dots, e_m, \pi e_{m+1}, \dots, \pi e_n\}\otimes k
\]
and when $s$ is odd the closed orbit is the orbit of the $k$-valued point
\[
\mathcal{F} = \text{span}_{O_{F_0}} \{-\pi^{-1}e_1,-e_1, \dots, -e_m, \pi e_{m+1}, \dots, \pi e_{n-1}\}\otimes k.
\]


Next, we consider the moduli scheme $\mathcal{M}$ over $O_F$, the \textit{naive splitting (or Krämer) model} as in \cite[Remark 14.2]{PR2} and \cite[Definition 4.1]{Kr}, whose points for an $O_F $-scheme $S$ are Zariski locally $\mathcal{O}_S$-direct summands $ \mathcal{F}_0 , \mathcal{F}_1 $ of ranks $s$, $n$ respectively, such that

\begin{enumerate}
    \item $ (0) \subset \mathcal{F}_0 \subset \mathcal{F}_1 \subset \Lambda_m  \otimes \mathcal{O}_S$;
    \item $\mathcal{F}_1 = \mathcal{F}^{\bot}_1$, i.e. $\mathcal{F}_1$ is totally isotropic for $( \,
, \, ) \otimes \mathcal{O}_S$;
    \item $(t + \pi) \mathcal{F}_1 \subset \mathcal{F}_0$;
    \item $(t - \pi)\mathcal{F}_0 = (0)$.
\end{enumerate}
The functor is represented by a projective $O_F $-scheme $ \mathcal{M}$. The scheme $ \mathcal{M}$ supports an action of $\calG$ and there is a $\calG$-equivariant projective morphism 
\[
\tau : \mathcal{M} \rightarrow {\rm M}^{\wedge}\otimes_{O} O_F \subset {\rm M}^{\rm naive}\otimes_{O} O_F
\]
which is given by $(\mathcal{F}_0,\mathcal{F}_1) \mapsto \mathcal{F}_1$ on $S$-valued points. (Indeed, we can easily see, as in \cite[Definition 4.1]{Kr}, that $\tau$ is well defined.) The morphism $\tau  : \mathcal{M} \rightarrow {\rm M}^{\wedge}\otimes_{O} O_F $ induces an isomorphism on the generic fibers (see \cite[Remark 4.2]{Kr}). In \S \ref{spl}, we will prove that $ \mathcal{M} $ is not flat for any signature $(r,s)$ and we will study the following variation:  
\[
{\rm M}^{\rm spl}:=\tau^{-1}({\rm M}^{\rm loc})=\calM\times_{{\rm M}^\wedge} {\rm M}^{\rm loc}.
\]
The closed subscheme ${\rm M}^{\rm spl} $ is the \textit{splitting model}. We will show that this model is smooth for the signature $(n-1,1) $ and has semi-stable reduction for the signatures $(n-2,2)$ and $(n-3,3)$. Moreover, for these signatures we will show that ${\rm M}^{\rm spl}$ has an explicit moduli-theoretic description.
 

\section{Two Affine Charts}\label{Two affine charts}

The goals of this section are to write down the equations that define the naive splitting model $\mathcal{M}$ in two open neighborhoods $_1\U_{r,s}$ and $_2\U_{r,s}$ for general signature $(r,s)$. Here, $_1\U_{r,s}$ is an affine chart around $(\calF_0, t\Lambda_m)$ (see \S \ref{AffineChart1}) and $_2\U_{r,s}$ is an affine chart around $(\calF_0, \Lambda')$, where $\Lambda'$ is defined in \S \ref{AffineChart2}. 


\subsection{An Affine Chart $_1\U_{r,s}$}\label{AffineChart1}
Recall that $n=2m= r+s$ with $s\leq r$. In this subsection, we are going to write down the equations that define $\calM$ in a neighborhood $_1\calU_{r,s}$ of $(\calF_0, t\Lambda_m)$ where


\[
\Lambda_m = \text{span}_{O_F} \{\pi^{-1}e_1, \dots, \pi^{-1}e_m, e_{m+1}, \dots, e_n\} =
\]
\[
\text{span}_{O_{F_0}} \{\pi^{-1}e_1, \dots, \pi^{-1}e_m, e_{m+1}, \dots, e_n,e_1, \dots, e_m, \pi e_{m+1}, \dots, \pi e_n\}.
\]
The matrix of $(\,,\, )$ in the latter basis is 
\[
\left[\ 
\begin{matrix}[c|c]
0_n & J_n  \\ \hline
-J_n & 0_n 
\end{matrix}\ \right],
\]
where 
    \[
    J_n = \left[\ 
\begin{matrix}[c|c]
0_m & -H_m  \\ \hline
H_m & 0_m 
\end{matrix}\ \right],
    \]
and $H_m$ is the unit antidiagonal matrix (of size $m$). Observe that $ J^2_n = -I_n.$ Also, since $ s \leq r$ and $ n =2m =r+s$ we get that $s\leq m$.

In order to find the explicit equations that describe $_1\U_{r,s}$, we use similar arguments as in the proof of  \cite[Theorem 4.5]{Kr} (see also \cite[\S 3]{Zac1}). In our case, we consider: 
\[  
   \mathcal{F}_1 =    \left[\ 
\begin{matrix}[c]
A\\ \hline
I_n  
\end{matrix}\ \right], \quad \mathcal{F}_0 = X =  \left[\ 
\begin{matrix}[c]
X_1\\ \hline
X_2  
\end{matrix}\ \right], 
\]
where $A$ is of size $ n \times n$, $X$ is of size $2n \times s$ and $X_1,X_2$ are of size $n\times s$; with the additional condition that $\mathcal{F}_0$ has rank $s$ and so $X$ has an invertible $s\times s$-minor. We also ask that $(\mathcal{F}_0,\mathcal{F}_1)$ satisfy the following four conditions:
\begin{enumerate}
            \begin{multicols}{2}
            \item $\mathcal{F}_1^{\bot} = \mathcal{F}_1$,  
            \item $ \mathcal{F}_0 \subset \mathcal{F}_1$,

            \columnbreak
            \item $(t - \pi)\mathcal{F}_0 = (0)$,
            
            \item $(t + \pi) \mathcal{F}_1 \subset \mathcal{F}_0$.

            \end{multicols}
        \end{enumerate}
Observe that
\[
M_t = \left[\ 
\begin{matrix}[c|c]
0_n & \pi_0I_n  \\ \hline
I_n & 0_n 
\end{matrix}\ \right]
\]
of size $2n \times 2n $ is the matrix giving multiplication by $t$.       

Condition (1), i.e. $\mathcal{F}_1$ is isotropic, translates to $A^{t} = -J_nAJ_n$. Condition (2), i.e., $ \mathcal{F}_0 \subset \mathcal{F}_1$ implies that the generators of $\calF_0$ can be expressed as a linear combination of the generators in $\calF_1$, which translates to
\[
 \exists \, Y : X =  \left[\ 
\begin{matrix}[c]
A\\ \hline
I_n  
\end{matrix}\ \right] \cdot Y, 
\]
where $Y$ is of size $ n\times s$. Thus, we have 
\[
\left[\ 
\begin{matrix}[c]
X_1\\ \hline
X_2  
\end{matrix}\ \right]  = \left[\ 
\begin{matrix}[c]
A\\ \hline
I_n  
\end{matrix}\ \right] \cdot Y  = \left[\ 
\begin{matrix}[c]
A   Y\\ \hline
  Y
\end{matrix}\ \right], 
\]
and so $ X_1 = AY$ and $X_2 = Y.$ Condition (3), i.e. $(t - \pi)\mathcal{F}_0 = (0)$, is equivalent to 
    \[
       M_t \cdot X =   \left[\ 
\begin{matrix}[c]
\pi X_1\\ \hline
\pi X_2
\end{matrix}\ \right],
    \]  
which amounts to 
\[
  \left[\ 
\begin{matrix}[c]
\pi_0 X_2\\ \hline
 X_1
\end{matrix}\ \right]= \left[\ 
\begin{matrix}[c]
\pi X_1\\ \hline
\pi X_2
\end{matrix}\ \right].
\]
Thus, $X_1 = \pi X_2$ which translates to $A Y = \pi Y $ by condition $(2)$.  
The last condition $(t + \pi) \mathcal{F}_1 \subset \mathcal{F}_0$ translates to 
\[
\exists \, Z  :  M_t \cdot \left[\ 
\begin{matrix}[c]
A   \\ \hline
I_n     
\end{matrix}\ \right]  + \left[\ 
\begin{matrix}[c]
\pi A   \\ \hline
\pi I_n     Y
\end{matrix}\ \right]  = X \cdot Z^t,
\]
where $Z$ is of size $n\times s$. This amounts to 
\[
  \left[\ 
\begin{matrix}[c]
\pi_0I_n + \pi A    \\ \hline
A + \pi I_n  
\end{matrix}\ \right] =  \left[\ 
\begin{matrix}[c]
X_1 Z^t   \\ \hline
X_2 Z^t
\end{matrix}\ \right] .
\]
From the above we get $A + \pi I_n  = X_2 Z^t$ which by condition (2) translates to $ A = Y Z^t-\pi I_n. $ Thus, $A$ can be expressed in terms of $Y,Z$. In addition, by condition (2) and in particular by the relations $ X_1 = AY$ and $X_2 = Y$ we deduce that the matrix $X$ is given in terms of $Y,Z$. Conversely, from $Y=X_2$ we get that the matrix $Y$ is given in terms of $A,X$ (Below we will also show that $Z$ can be expressed in terms of $A,X$).  

Observe from $X_1 = \pi X_2$ that all the entries of $X_1$ are in the maximal ideal and thus a minor involving entries of $X_1$ cannot be a unit. Recall that the matrix $X$ has an invertible $s\times s$-minor and $X_2=Y$ from condition (2). Thus, the matrix $Y$ has a $s\times s$-minor. Let 
\[
Y = \left[\begin{array}{ccc} 
         y_{1,1} &\cdots  &  y_{1,s}  \\
         \vdots &  \ddots & \vdots   \\
          y_{n,1} & \cdots &  y_{n,s}   
         \end{array}\right] 
         , \quad
    Z = \left[\begin{array}{ccc} 
         z_{1,1} &\cdots  &  z_{1,s}  \\
         \vdots &  \ddots & \vdots   \\
          z_{n,1} & \cdots &  z_{n,s}    
         \end{array}\right] =  \left[\begin{array}{ccc} 
        \mathbf{z}_1 \ \cdots
         \  \mathbf{z}_s
         \end{array}\right],
\]
and 
\begin{equation}\label{Minors1}
Y'=\left[\begin{array}{ccc} 
         y_{i_1,1} &\cdots  &  y_{i_1,s}  \\
         \vdots &  \ddots & \vdots   \\
          y_{i_s,1} & \cdots &  y_{i_s,s}   
         \end{array}\right] = I_s
\end{equation}
where $1 \leq i_k \leq n$ for $k = 1, \dots,s$. Here, we want to mention how we can express $Z$ in terms of $A,X$. From the above we have $Y Z^t=A+ \pi I_n$ and 
\[
Y Z^t = \left[\begin{array}{c}
         \sum^s_{e=1} y_{1,e} \mathbf{z}^t_e  \\
          \vdots   \\
            \sum^s_{e=1} y_{n,e} \mathbf{z}^t_e  
 \end{array}\right].
\] 
By using (\ref{Minors1}), for any $1 \leq k \leq s$, the $i_k$-th row of the matrix $Y Z^t $ gives $\mathbf{z}^t_k $. Since $Y Z^t=A+ \pi I_n$, we can easily see that the matrix $Z$ can be expressed in terms of $A,X$.

We replace $A$ by $Y Z^t-\pi I_n$. Hence, conditions ($1$) and ($3$) are equivalent to
\begin{eqnarray}
    Z Y^t &=& -J_nY Z^tJ_n, \\
    Y Z^t Y &=& 2 \pi Y. 
\end{eqnarray} 
From the above we deduce that

\begin{Lemma}\label{Lemma1}
An affine chart $_1\U_{r,s}$ of $\mathcal{M}$ around $(\mathcal{F}_0, t \Lambda_m)$ is given by $\Spec O_F[Y , Z ]/\mathscr{I} $ where 
\[ \mathscr{I} = (Y'-I_s,\ Z Y^t +J_nY Z^tJ_n,\ Y Z^t Y - 2 \pi Y)
\]
and $Y'$ is a submatrix composed of $s$ rows of $Y$ (see the relation (\ref{Minors1})).\qed
\end{Lemma}

Note that the affine chart $_1\U_{r,s}$ of $\mathcal{M}$ depends on the matrix $Y'$ we fixed (see Remark \ref{notiso1}).
Our next goal is to give a simplification of the above equations (see Theorem \ref{Simpl.1}). We first study the relation $Z Y^t =-J_nY Z^tJ_n$. Consider the matrix $P = (p_{i,j})$ of size $n=2m$ with $P^t=-J_nPJ_n$. By a direct calculation we get
\[
p_{i,j}=p_{n+1-j,n+1-i}
\]
for $1\leq i, j\leq m$ or $m+1\leq i,j\leq n$, and
\[
p_{i,j}=-p_{n+1-j,n+1-i}
\]
for $1\leq i\leq m$ and $m+1\leq j\leq n$, or $1\leq j\leq m$ and $m+1\leq i\leq n$. By using this observation and by setting $P = Y Z^t$ the relation $ Z Y^t=-J_nY Z^tJ_n$ gives 
\begin{equation}\label{Eq1zy}
\sum^s_{k=1} z_{i,k}y_{j,k} = \pm \sum^s_{k=1} z_{n+1-j,k}y_{n+1-i,k}
\end{equation}
for $1\leq i,j \leq n$, where the different signs $\pm$ depends on the position of $i,j$ as above. Consider the following 2 cases:

\textit{Case 1:} Assume that $ 1\leq i_k \leq m$ for all $k \in \{1,\dots,s\}.$ By setting $i = n+1-i_q $ and $j= i_p$ the equation (\ref{Eq1zy}) gives $ z_{n+1-i_q,p} = -z_{n+1-i_p,q} .$ In particular, if $p=q$ we get $z_{n+1-i_p,p} = 0 $ for all $p \in \{1,\dots,s\}$. By setting $  d_{p,q}  = z_{n+1-i_p,q}$ we get the matrix $D = (d_{p,q})_{p,q}$ of size $s \times s$ such that 
\begin{equation}
D = \left[\begin{array}{ccc} 
        d_{1,1} &\cdots  &  d_{s,1}  \\
         \vdots &  \ddots & \vdots   \\
          d_{1,s} & \cdots &  d_{s,s}    
         \end{array}\right] = \left[\begin{array}{ccc} 
        z_{n+1-i_1,1} &\cdots  &  z_{n+1-i_s,1}  \\
         \vdots &  \ddots & \vdots   \\
          z_{n+1-i_1,s} & \cdots &  z_{n+1-i_s,s}    
         \end{array}\right]  .
\end{equation}
The matrix $D$ is skew-symmetric, i.e. $D = -D^t.$ Next, by setting $j = i_k$ in  (\ref{Eq1zy}) we obtain
\[
z_{i,k} =
\left\{
	\begin{array}{ll}
		\sum^s_{e=1} d_{k,e}y_{n+1-i,e}  & \mbox{if } 1 \leq i \leq m,  \\
		-\sum^s_{e=1} d_{k,e}y_{n+1-i,e} & \mbox{if } m+1\leq i \leq n.
	\end{array}
\right.
\]
Thus, we have $ Z = \wti{Y} \cdot D$ where 
\[
\wti{Y} = \left[\begin{array}{ccc} 
         y_{n,1} & \cdots & y_{n,s}  \\
         \vdots & \ddots &  \vdots   \\
          y_{m+1,1}  & \cdots &   y_{m+1,s}   \\ \hline
          -y_{m,1} & \cdots &  -y_{m,s}  \\
         \vdots & \ddots &  \vdots   \\
          -y_{1,1}  & \cdots &  -y_{1,s}   \\
   \end{array}\right].
\]
Therefore, in this case the equations (\ref{Eq1zy}) give $ Z = \wti{Y} \cdot D $. 

\textit{Case 2:} In general, we let $ 1 \leq t \leq s$ and assume that 
\[
1\leq i_1, \dots, i_t \leq m \quad  \text{and} \quad   m+1\leq i_{t+1}, \dots, i_s \leq n.
\]
By setting $d_{p,q} = z_{n+1-i_p,q}$ and using the same arguments as above, we get the matrix $D = (d_{p,q})_{p,q}$ of size $s \times s$ such that 
\[
d_{p,q}=-d_{q,p}
\]
for $1\leq p, q \leq t$ or $t+1\leq p,q\leq n$, and
\[
d_{p,q}=d_{q,p}
\]
for $1\leq p\leq t$ and $t+1\leq q\leq n$, or $1\leq q\leq t$ and $t+1\leq p\leq n$. Next, by setting $j = i_k$ in  (\ref{Eq1zy}) for $1\leq k \leq t$ we obtain
\[
z_{i,k} =
\left\{
	\begin{array}{ll}
		\sum^s_{e=1} d_{k,e}y_{n+1-i,e}  & \mbox{if } 1 \leq i \leq m,  \\
		-\sum^s_{e=1} d_{k,e}y_{n+1-i,e} & \mbox{if } m+1\leq i \leq n.
	\end{array}
\right.
\] 
Also, by setting $j = i_k$ for $t+1\leq k \leq s$ we get 
\[
z_{i,k} =
\left\{
	\begin{array}{ll}
		-\sum^s_{e=1} d_{k,e}y_{n+1-i,e}  & \mbox{if } 1 \leq i \leq m,  \\
		\sum^s_{e=1} d_{k,e}y_{n+1-i,e} & \mbox{if } m+1\leq i \leq n.
	\end{array}
\right.
\] 
From the above we have $ Z = \wti{Y} \cdot \wti{D}$ where $\wti{D} = (\wti{d}_{i,j})_{i,j}$ is the $s \times s$ matrix defined as
\[
\wti{D} = \left[\begin{array}{ccc|ccc} 
         d_{1,1} & \cdots & d_{t,1} & -d_{t+1,1} &  \cdots & -d_{s,1}\\
         \vdots & \ddots &  \vdots  &  \vdots & \ddots &  \vdots   \\
         d_{1,s} & \cdots & d_{t,s} & -d_{t+1,s} &  \cdots & -d_{s,s}\\
   \end{array}\right].
\]
It is easy to see that the matrix $ \wti{D}$ is skew symmetric, i.e. $\wti{d}_{i,j} = - \wti{d}_{j,i}.$ In particular, we have the following 4 cases: when $1 \leq i,j \leq t$ we have $ \wti{d}_{i,j} = d_{i,j} = -d_{j,i} = - \wti{d}_{j,i}$. For $1 \leq i \leq t$ and $t+1 \leq j \leq s $ we have $ \wti{d}_{i,j} = -d_{i,j} = -d_{j,i} = - \wti{d}_{j,i}$. When $1 \leq j \leq t$ and $t+1 \leq i \leq s $ we get $ \wti{d}_{i,j} = d_{i,j} = d_{j,i} = - \wti{d}_{j,i}$. Lastly, for $t+1 \leq i,j \leq s$ we obtain $ \wti{d}_{i,j} = -d_{i,j} = d_{j,i} = - \wti{d}_{j,i}$. 

Combining the above cases, for any $1\leq i_1,\dots, i_s\leq n$, the equations (\ref{Eq1zy}) show that there exists a skew symmetric matrix $\mathcal{D}$ of size $s \times s$ 
such that
\[
Z = \wti{Y} \cdot \mathcal{D}.
\]
By abuse of notation, write $ D$ for $ \mathcal{D}$. Next, define $ Q : =\wti{Y}^t \cdot Y $
of size $ s \times s.$ Note that $ Q^t = -Q.$ We are now ready to show:

\begin{Theorem}\label{Simpl.1}
For any signature $(r,s)$, there is an affine chart $_1\U_{r,s} \subset \mathcal{M}$ which is isomorphic to $\Spec O_F[ D, Y]/I $ where 
\[
I = \left(Y'-I_s,\ D+D^t,\ D\cdot Q + 2 \pi I_s\right)
\]
for some choice of $Y'$.
\end{Theorem}
\begin{proof}
Recall from Lemma \ref{Lemma1} that $\U = \Spec O_F[Y , Z ]/\mathscr{I} $ where 
\[
\mathscr{I} = (Y'-I_s,\ Z Y^t +J_nY Z^tJ_n,\ Y Z^t Y - 2 \pi Y).
\] 
From the above discussion we have that the relation $Z Y^t =-J_nY Z^tJ_n $ 
 gives $ Z = \wti{Y} \cdot D$. Thus, it suffices to focus on the relation $ Y Z^t Y - 2 \pi Y=0$. Observe that $ Y \cdot Z^t \cdot Y =-  Y \cdot  D \cdot  Q.$ Thus,
\[
Y Z^t Y - 2 \pi Y = -Y\cdot( D \cdot  Q + 2 \pi I_s)=0.
\]
Now, by using the relation with the minors (\ref{Minors1}) we can easily see that the relation $Y Z^t Y - 2 \pi Y=0$ gives $D \cdot  Q + 2 \pi I_s=0.$ Therefore $_1\U_{r,s}$ is isomorphic to 
\[
\Spec \frac{O_F[D,Y]}{\left(Y'-I_s,\ D+D^t, \ D\cdot Q + 2 \pi I_s\right)}.
\]
\end{proof}
In particular, when the signature $(r,s)=(n-2,2)$, we have
\[
D=\left[\begin{array}{cc}
0&d\\
-d&0
\end{array}\right],\quad
Q=\left[\begin{array}{cc}
0&q\\
-q&0
\end{array}\right]
\]
where $d=d_{2,1}=-d_{12}$ and $q=\sum^m_{j=1}y_{n+1-j,1}y_{j,2} -  \sum^n_{j=m+1}y_{n+1-j,1}y_{j,2}$.

\begin{Corollary}\label{CorChart1}
When $(r,s)=(n-2,2)$, there is an affine chart $_1\mathcal{U}_{r,s} \subset \mathcal{M}$ which is isomorphic to $V(I) = \Spec O_F[d,(y_{i,1})_{1 \leq i \leq n},(y_{i,2})_{1 \leq i \leq n} ]/ I$, where 
\[
I= ( y_{i_0,1}-1,\,\, y_{j_0,2} -1 ,\,\,y_{j_0,1},\,y_{i_0,2},\,\,  d(\sum^m_{j=1}y_{n+1-j,1}y_{j,2} -  \sum^n_{j=m+1}y_{n+1-j,1}y_{j,2}) 
-2 \pi)
\]
for some $1\leq i_0,j_0 \leq n$.
\end{Corollary}
\begin{Remark}\label{notiso1}{\rm
Recall that $_1\U_{r,s}$ is an affine chart around $(\calF_0, t\Lambda_m)$. Note that for different choices of $Y'$ we get different equations from $D\cdot Q  =- 2 \pi I_s$ and so different affine charts $_1\U_{r,s}$ which in general are not isomorphic. For instance, in Corollary \ref{CorChart1}, there are two different $V(I)$ up to isomorphism when we choose different places for $i_0, j_0$ (see \S \ref{Evensemistable} for more details). In Proposition \ref{SemiStabCase1}, we show that for any such a choice the affine scheme $V(I)$ has semi-stable reduction over $O_F$. 
}\end{Remark}

\subsection{An Affine Chart $_2\U_{r,s}$}\label{AffineChart2}\label{2ndChart}

Keep the same notation as above. As in \cite{PR}, we set $\Lambda_m=\Lambda'\oplus\Lambda''$, where 
\[
\begin{array}{l}
\Lambda'=\tspan_{O_{F_0}}\{-\pi^{-1}e_1, -e_1,\dots,-e_m,\pi e_{m+1},\dots, \pi e_{n-1}\},\\
\Lambda''=\tspan_{O_{F_0}}\{e_n, \pi e_n, -\pi^{-1}e_2, \dots, -\pi^{-1}e_m, e_{m+1}, \dots, e_{n-1}\}.
\end{array}
\]
In this subsection, we are going to write down the equations that define $\calM$ in a neighborhood $_2\calU_{r,s}$ of $(\calF_0, \Lambda')$. 

The matrix of the form $(\ ,\ )$ in this basis is:
\[\left[
\begin{array}{c|c}
	0_n & S\\ \hline
	S^t & 0_n
\end{array}
\right],\]
where $S$ is the skew-symmetric  matrix of size $n$ given by:
\[
S=\left[
\begin{array}{c|c}
	J_2^t & 0\\ \hline
	0 & J_{n-2}
\end{array}
\right].
\]
Observe that $S^2=-I_n$ since $J_n^2=-I_n$. For any point $(\calF_0,\calF_1)\in \calM$, we ask that $ (\calF_0,\calF_1)$ satisfy the following conditions:
\begin{enumerate}
	\begin{multicols}{2}
            \item $\mathcal{F}_1^{\bot} = \mathcal{F}_1$,  
            \item $ \mathcal{F}_0 \subset \mathcal{F}_1$,
            
          \columnbreak
            \item $(t - \pi)\mathcal{F}_0 = (0)$,
            
            \item $(t + \pi) \mathcal{F}_1 \subset \mathcal{F}_0$.

            \end{multicols} 
\end{enumerate}
Similar to \S \ref{AffineChart1}, we consider
\[
\calF_1=\left[\begin{array}{c}
	I_n\\ \hline
	A
\end{array}
\right],\quad
\calF_0=X=\left[\begin{array}{c}
	X_1\\ \hline
	X_2
\end{array}
\right],
\]
where $I_n, A$ are $n\times n$-matrices, and $X_1, X_2$ are $n\times s$-matrices. Recall that $\calF_0$ has rank $s$, so $X$ has a no-vanishing $s\times s$ minor. Write down the matrix $A$ by 
\[
A=\left[\begin{array}{c|c}
	T & B\\ \hline
	C & Y
\end{array}
\right],
\]
where $T$ is a $2\times 2$-matrix, $B$ (resp. $C$) is a $2\times (n-2)$-matrix (resp. $(n-2)\times 2$-matrix) and $Y$ is of size $(n-2)\times (n-2)$. Then condition (1) is equivalent to $A^t S=SA$. This translates to 
\[
T^t=-J_2TJ_2,\quad C=J_{n-2}B^tJ_2,\quad Y^t=-J_{n-2}YJ_{n-2}.
\]
Note that the first equation immediately gives $T=\text{diag}(x,x)$ for some variable $x$, and the second equation shows that $C$ is given in terms of $B$. Condition (2) is equivalent to
\[
\exists M:  \left[\begin{array}{c}
	X_1\\ \hline
	X_2
\end{array}
\right]=
\left[\begin{array}{c}
	I_n\\ \hline
	A
\end{array}
\right]\cdot M,
\]
where $M$ is a $n\times s$-matrix. This amounts to $X_1=M, X_2=AM$. Let 
\[
M=\left[\begin{array}{c}
	M_1\\ \hline
	M_2
\end{array}
\right],
\]
where $M_1$ is of size $2\times s$ and $M_2$ is of size $(n-2)\times s$. We can see that $X_1, X_2$ can be expressed in terms of $A, M_1, M_2$. More precisely, we get
\[
X_1=M=\left[\begin{array}{c}
	M_1\\ \hline
	M_2
\end{array}
\right],\quad
X_2=AM=\left[\begin{array}{c}
	xM_1+BM_2\\ \hline
	CM_1+YM_2
\end{array}
\right].
\]
The map $t: \Lambda_m\rightarrow \Lambda_m$ is represented by the matrix
\[
M_t=\left[\begin{array}{cccc}
	t_0&&&\\
	&&& I_{n-2}\\
&& t_0 &\\
& \pi_0 I_{n-2} &&
\end{array}
\right],
\] 
where $t_0$ is the $2\times 2$-antidiag matirx $\left[\begin{array}{cc}
	& \pi_0\\
	1& 
\end{array}
\right]$. Condition (3) translates to 
\[
M_t\cdot\left[ \begin{array}{c}
	X_1\\ \hline
	X_2
\end{array}\right]=\left[ \begin{array}{c}
	\pi X_1\\ \hline
	\pi X_2
\end{array}\right].
\]
By setting $B_i$ as the $i$-th row of $B$ and by using condition (2), the above relation amounts to
\[
M_1=\left[\begin{array}{c}
	\pi M_0\\ \hline
	M_0
\end{array}
\right],\quad B_1M_2=\pi B_2M_2,\quad CM_1+YM_2=\pi M_2,
\]
where $M_0$ is the second row of $M_1$. Finally, condition (4) is equivalent to 
\[
(M_t+\pi I_{2n})\cdot \left[\begin{array}{cc}
	I_2 &\\
	& I_{n-2}\\
	T & B\\
	C & Y
\end{array}
\right]=\left[\begin{array}{c}
	M_1\\
	M_2\\
	xM_1+BM_2\\
	CM_1+YM_2
\end{array}
\right]\cdot 
\left[\begin{array}{cc}
	Z_1^t & Z_2^t
\end{array}
\right],
\]
where $Z_1$ is of size $2\times s$ and $Z_2$ is of size $(n-2)\times s$. This, in turn, gives
\[
\begin{array}{lll}
	M_0Z_1^t=[1\ \pi], & M_0Z_2^t=0, & M_2Z_1^t=C,\\
	M_2Z_2^t=\pi I_{n-2} +Y, & Y^2= \pi_0 I_{n-2}, & B_1=B_2Y.
\end{array}
\]

From the above equations, we can see that $A$ is given in terms of $x, C, Y$ by condition (1). Note that $C=M_2Z_1^t$, and $Y=M_2Z_2^t-\pi I_{n-2}$ by condition (4). Thus, the matrix $A$ is given in terms of $M, Z_1, Z_2$ and a variable $x$. Meanwhile, we have $X_1=M, X_2=AM$ by condition (2), so that $X_1, X_2$ are also given in terms of $M, Z_1, Z_2$. On the other hand, we have $M=X_1$ and we claim that $Z_1, Z_2$ can be expressed in terms of $A, X_1, X_2$ (we will show this later in this section).

Before we move on, let us simplify our equations first. By replacing $Y=M_2Z_2^t-\pi I_{n-2}$, the equation $Y^t=-J_{n-2}YJ_{n-2}$ from condition (1) amounts to $(M_2Z_2^t)^t=-J_{n-2}(M_2Z_2^t)J_{n-2}$, and the equations that involve $Y$ from condition (3), (4) translate to
\[
CM_1+(M_2Z_2^t)M_2=2\pi M_2,\quad
(M_2Z_2^t)^2=2\pi(M_2Z_2^t),\quad
B_1+\pi B_2=B_2(M_2Z_2^t).
\] 
Finally, we replace $C$ by $M_2Z_1^t$, and deduce that

\begin{Lemma}\label{lm44}
	The affine chart $_2\U_{r,s}$ of $\calM$ around $(\calF_0, \Lambda')$ is given by 
	\[
	\Spec O_F[x, B, M_0, M_2, Z_1, Z_2]/\mathscr{I},
	\]
	where $\mathscr{I}$ is the ideal generated by the entries of following equations
	\[
\begin{array}{lll}
M_2Z_1^t=J_{n-2}B^tJ_2, & (M_2Z_2^t)^t=-J_{n-2}(M_2Z_2^t)J_{n-2}, & (M_2Z_2^t)^2=2\pi(M_2Z_2^t),\\
B_1M_2=\pi B_2M_2, &B_1+\pi B_2=B_2(M_2Z_2^t), & M_0Z_2^t=0,\\
M_0Z_1^t=[1\ \pi], &(M_2Z_1^t)M_1+(M_2Z_2^t)M_2=2\pi M_2.  &  	
\end{array}
\]
Here $B_i$ is the $i$-th row of $B$ for $i=1,2$.\qed
\end{Lemma}


\begin{Remark}{\rm
Condition (3) also gives the equations: 
 \[
 BC=0,\quad Ct_0+YC=0.
 \]
In \cite[\S 5.3]{PR}, the authors show that these two equations are automatically true if $B,C,Y$ satisfy
\[
\begin{array}{ll}
	Y^2=\pi_0 I_{n-2}, & Y^t=-J_{n-2}YJ_{n-2},\\
	B_1=B_2Y, & C=J_{n-2}B^tJ_2.
\end{array}
\]
}\end{Remark} 


\subsubsection{The case $(r,s)=(n-1,1)$} In this and the next subsection, our goal is to give a simplification of the equations in Lemma \ref{lm44}. We first consider the affine chart $_2\mathcal{U}_{n-1,1} \subset \mathcal{M}$ for the signature $(n-1,1)$. 
Note that $M_0$ is of size $1\times 1$ in this case. Since
\[
M_0 Z_1^t=[1\ \pi]\neq 0,
\]
we get $M_0\neq 0$. Without loss of generality, we assume $M_0=1$ as rank($\calF_0$)=1. 

\begin{Theorem}\label{Prop.s=1}
	When the signature $(r,s)=(n-1,1)$, the affine chart $_2\calU_{n-1,1}$ is smooth and isomorphic to $\mathbb{A}_{O_F}^{n-1}$.
\end{Theorem}

\begin{proof}
From $M_0Z_1^t=[1\ \pi]$, $M_0Z_2^t=0$ we obtain
\[
Z_1^t=\left[\begin{array}{cc}
	1& \pi\\ 
\end{array}
\right],\quad
Z_2^t=0.
\]
Let $M_2=[a_1\ \cdots\ a_{n-2}]^t$ be an $(n-2)\times 1$ vector. Then $C=M_2Z_1^t$ amounts to:
\[
C=\left[\begin{array}{c|c}
	a_1 &\pi a_1\\ 
	\vdots & \vdots \\
	a_{m-1} &\pi a_{m-1}\\ \hline
	a_{m} &\pi a_{m}\\
	\vdots & \vdots \\
	a_{n-2}& \pi a_{n-2}
\end{array}
\right].
\]
Also, from $B=J_2C^tJ_{n-2}$ we get
\[
B=\left[\begin{array}{ccc|ccc}
	-\pi a_{n-2} &\cdots &-\pi a_m& \pi a_{m-1}& \cdots &\pi a_1 \\ \hline
	 a_{n-2} &\cdots & a_m& -a_{m-1}& \cdots & -a_1
\end{array}\right].
\] 
It is easy to see that the relation $Y=M_2Z_2^t-\pi I_{n-2}$ gives $Y=-\pi I_{n-2}$ and $B_1+\pi B_2=B_2(M_2Z_2^t)=0$. Finally, we check that 
\[
\begin{array}{l}
	B_1M_2=\pi B_2M_2=\pi(\sum_{i=1}^{m-1}a_ia_{n-1-i}-\sum_{i=1}^{m-1}a_ia_{n-1-i})=0,\\
	CM_1+(M_2Z_2^t)M_2=2\pi\cdot [a_1\ \cdots\ a_{n-2}]^t=2\pi M_2.
\end{array}
\]
From the above,, the conditions on the $a_i$ are automatically satisfied.  We deduce that the affine chart $_2\calU_{n-1,1}$ is isomorphic to 
\[
\Spec O_F[x, a_1, \dots, a_{n-2}]=\mathbb{A}_{O_F}^{n-1}.
\]
\end{proof}


\subsubsection{The cases for $s\geq 2$}\label{4.2.2} In this subsection, we consider the explicit equations of $_2\calU_{r,s}$ for the signature $(r,s)$ with $s\geq 2$. Observe that $\calF_1=\{v+A(v)\mid v\in\Lambda'\otimes\calO_S \}$. In the special fiber (where $\calO_S=k$), we have $t\calF_1\subset \calF_0$, and $t(\Lambda'\otimes k)=\text{span}_{O_{F_0}}\{-e_1\}$, which implies $-e_1+tA(v)\in\calF_0$. Since
\[
\calF_0=X=\left[\begin{array}{c}
	M_1\\ 
	M_2 \\ \hline
	xM_1+BM_2\\ 
	CM_1+YM_2
	\end{array}
\right],
\]
we can always assume that $-e_1+tA(v)$ is the first 
generating element of
$\calF_0$. In the matrix language, this is equivalent to assume that the second row of $X$, i.e. $M_0$, is equal to $[1\ 0 \cdots 0]$. Thus, we have
\[
M_1=\left[\begin{array}{cccc}
	\pi & 0&\cdots& 0\\ 
	1& 0&\cdots &0 
	\end{array}
\right].
\]
By using $M_0Z_1^t=[1 \ \pi]$, $M_0Z_2^t=0$, we set
\[
Z_1^t=\left[\begin{array}{cc}
	1& \pi \\ 
	z_{1,2} & z_{2,2}\\
	\vdots &\vdots\\
	z_{1,s} & z_{2,s}
	\end{array}
\right], \quad
Z_2^t=\left[\begin{array}{ccc}
	0&\cdots & 0 \\ 
	z_{1,2}'&\cdots & z_{n-2,2}'\\
	\vdots & \ddots & \vdots \\
	z_{1,s}'&\cdots & z_{n-2,s}'
	\end{array}
\right],
\]
and
\[
M_2=\left[\begin{array}{cccc}
\mathbf{a_1} & \cdots &\mathbf{a_s}
\end{array}\right]=
\left[\begin{array}{ccc}
	a_{1,1}& \cdots & a_{1,s} \\ 
	\vdots & \ddots &\vdots \\
	a_{n-2,1} & \cdots & a_{n-2,s}
	\end{array}
\right].
\]
We claim that $Z_1, Z_2$ can be expressed in terms of $A, X_1, X_2$. Observe from $CM_1+YM_2=\pi M_2$ that all the entries of $CM_1+YM_2$ are in the maximal ideal, and thus a minor involving entries of $CM_1+YM_2$ cannot be a unit. Similarly, the entries of the first row of $xM_1+BM_2$ are also in the maximal ideal. By replacing the order of basis $\{\pi e_m,\dots, \pi e_n\}$ if necessary, we can assume that the matrix $M_2$ has an invertible $(s-1)\times (s-1)$-minor. Let $M_2'$ be a submatrix of $M_2$ composing $(s-1)$ rows along with the corresponding $s$ columns. We set
\[
M_2'=\left[\begin{array}{cccc}
	a_{i_2,1}& a_{i_2,2} &\cdots & a_{i_2,s} \\ 
	\vdots & \vdots &\ddots &\vdots \\
	a_{i_s,1} & a_{i_s,2}&\cdots  & a_{i_s,s}
	\end{array}
\right]=
\left[\begin{array}{cccc}
	0& 1 & \cdots & 0 \\ 
	\vdots & \vdots & \ddots &\vdots \\
	0 & 0 & \cdots & 1
	\end{array}
\right]
\]
for $1\leq i_2, \cdots, i_s\leq n-2$. Consider the matrix $C=M_2Z_1^t$. We have
\[
C=\left[\begin{array}{cc}
\mathbf{a_1}+z_{1,2}\mathbf{a_2}+\cdots +z_{1,s}\mathbf{a_s} & 
\pi \mathbf{a_1}+z_{2,2}\mathbf{a_2}+\cdots +z_{2,s}\mathbf{a_s}
\end{array}\right].
\]
By setting $i=i_2, \dots , i_s$, it is easy to see that $Z_1$ can be expressed in terms of $C$. Similarly, $Z_2$ can be expressed in terms of $Y$ as $M_2Z_2^t=\pi I_{n-2}+Y$. Thus, $Z_1, Z_2$ are given in terms of $A, X_1, X_2$.

Set 
\[
D=\left[\begin{array}{ccc}
d_{2,2}&\cdots & d_{2,s}\\
\vdots &\ddots &\vdots\\
d_{s,2}&\cdots & d_{s,s}
\end{array}\right],\quad
Q=\left[\begin{array}{ccc}
q_{2,2}&\cdots & q_{2,s}\\
\vdots &\ddots &\vdots\\
q_{s,2}&\cdots & q_{s,s}
\end{array}\right],
\]
where $q_{i,j}=\sum_{k=1}^{m-1}(a_{k,i}a_{n-1-k,j}-a_{n-1-k,i}a_{k,j})$. The matrix  $Q$ depends on $M_2$. Note that $Q^t=-Q$ by definition.

\begin{Theorem}\label{Simpl.2}
	For any signature $(r,s)$ with $s\geq 2$, there is an affine chart $_2\calU_{r,s}\subset \calM$ which is isomorphic to 
	\[
	\mathbb{A}_{O_F}^{s-1}\times \Spec O_F[D, M_2]/I,
	\]
	where
	\[
I=(M_2'-[\mathbf{0}\mid I_{s-1}],\ D+D^t,\ D\cdot Q+2\pi I_{s-1})
\]
for some choice of $M'_2$. $[\mathbf{0}\mid I_{s-1}]$ is of size $(s-1)\times s$, with the first column $0$.
\end{Theorem}

\begin{proof}
The proof is similar to the proof of Theorem \ref{Simpl.1}. Consider $(M_2Z_2^t)^t=-J_{n-2}(M_2Z_2^t)J_{n-2}$ first. We claim that $Z_2^t$ depends on $M_2$ and $D$. Since $M_2Z_2^t=(\sum_{k=2}^{s}a_{i,k}z_{j,k}')_{ij}$, the relation $(M_2Z_2^t)^t=-J_{n-2}(M_2Z_2^t)J_{n-2}$ amounts to 
\[
a_{i,2}z_{j,2}'+\cdots+a_{i,s}z_{j,s}'=\pm(a_{n-1-j,2} z_{n-1-i,2}'+\cdots +a_{n-1-j,s}z_{n-1-i,s}'),
\]
where the different sign $\pm$ depends on the position of $i_2,\dots, i_s$. We obtain
\[
z_{n-1-i_2,2}'=z_{n-1-i_3,3}'=\dots =z_{n-1-i_s,s}'=0,
\]
by letting $i=i_k$, $j=n-1-i_k$ for $k=2,\dots ,s$. To write down $Z_2^t$ explicitly, we need to consider the positions of $i_2,\dots ,i_s$.

{\it Case 1:} Assume that $1\leq i_2,\dots ,i_s\leq m-1$. In this case, we have $z_{n-1-i_p,q}'=-z_{n-1-i_q,p}'$ by letting $i=i_{p}$, $j=n-1-i_{q}$. Set $d_{p,q}=z_{n-1-i_p,q}'$. We obtain
\[
\left[\begin{array}{ccc}
z_{n-1-i_2,2}'& \cdots & z_{n-1-i_2,s}'\\
\vdots &\ddots & \vdots\\
z_{n-1-i_s,2}'& \cdots & z_{n-1-i_s,s}'
\end{array}\right]
=
\left[\begin{array}{ccc}
d_{2,2}&\cdots & d_{2,s}\\
\vdots &\ddots &\vdots\\
d_{s,2}&\cdots & d_{s,s}
\end{array}\right]=D,
\]
where $d_{p,q}=-d_{q,p}$ for $2\leq p,q\leq s$. This implies that $D$ is skew-symmetric.

For $i=i_2$, we have
\[
z_{j,2}'=\pm(a_{n-1-j,2}d_{2,2}+\cdots +a_{n-1-j,s}d_{2,s}).
\]
Since $1\leq i_2\leq m-1$, the sign $\pm$ is positive when $1\leq j\leq m-1$, and is negative when $m\leq j\leq n-2$. Similar to $i=i_k$ for $k=3,\dots, s$. We have
\[
Z_2^t=
\left[\begin{array}{ccc}
0&\cdots &0\\
d_{2,2}&\cdots & d_{2,s}\\
\vdots &\ddots &\vdots\\
d_{s,2}&\cdots & d_{s,s}
\end{array}\right]\cdot
\left[\begin{array}{ccc|ccc}
	a_{n-2,2}&\cdots & a_{m,2} & -a_{m-1,2}&\cdots & -a_{1,2}\\ 
	\vdots & \ddots & \vdots &\vdots &\ddots &\vdots\\
	a_{n-2,s}&\cdots & a_{m,s} & -a_{m-1,s} & \cdots &-a_{1,s}
\end{array}
\right].
\]

{\it Case 2:} In general, we assume that
\[
	1\leq i_2,\dots, i_t\leq m-1 \quad \text{and} \quad 
	m\leq i_{t+1},\dots, i_s\leq n-2. 
\]
Set $d_{p,q}=z_{n-1-i_p,q}'$. By letting $i=i_p, j=n-1-i_q$, we obtain
\[
\begin{array}{l}
d_{p,q}=-d_{q,p}, \quad \text{for}~ 2\leq p,q\leq t, ~\text{or}~ t+1\leq p,q\leq s,\\
d_{p,q}=d_{q,p}, \quad \text{for}~ 2\leq p\leq t, t+1\leq q\leq s, ~\text{or}~ t+1\leq p\leq s, 2\leq q\leq t.
\end{array}
\]
Thus, $z_{j,k}'=\pm(a_{n-1-j,k}d_{k,2}+\cdots +a_{n-1-j,s}d_{k,s})$, where the sign $\pm$ is positive if $1\leq j\leq m$, $2\leq k\leq t$, or $m\leq j\leq n-2$, $t+1\leq k\leq s$, and negative otherwise. Therefore, $Z_2^t$ can be expressed as
\[
Z_2^t=
\left[\begin{array}{ccc}
0&\cdots &0\\
d_{2,2}&\cdots & d_{2,s}\\
\vdots &\ddots &\vdots\\
d_{t,2}&\cdots & d_{t,s}\\ \hline
-d_{t+1,2}&\cdots & -d_{t+1,s}\\
\vdots &\ddots &\vdots\\
-d_{s,2}&\cdots & -d_{s,s}
\end{array}\right]\cdot
\left[\begin{array}{ccc|ccc}
	a_{n-2,2}&\cdots & a_{m,2} & -a_{m-1,2}&\cdots & -a_{1,2}\\ 
	\vdots &\ddots & \vdots &\vdots &\ddots &\vdots\\
	a_{n-2,s}&\cdots & a_{m,s} & -a_{m-1,s} & \cdots &-a_{1,s}
\end{array}
\right].
\]
Set
\[
\wti{D}=[\tilde{d}_{p,q}]_{2\leq p,q\leq s}=\left[\begin{array}{ccc}
d_{2,2}&\cdots & d_{2,s}\\
\vdots &\ddots &\vdots\\
d_{t,2}&\cdots & d_{t,s}\\ \hline
-d_{t+1,2}&\cdots & -d_{t+1,s}\\
\vdots &\ddots &\vdots\\
-d_{s,2}&\cdots & -d_{s,s}
\end{array}\right].
\]
For $2\leq p,q\leq t$, we have $\tilde{d}_{p,q}=d_{p,q}=-d_{q,p}=-\tilde{d}_{q,p}$. Similarly, it is easy to see that $\tilde{d}_{p,q}=-\tilde{d}_{q,p}$ for all the other positions of $p,q$. Hence, the matrix $\wti{D}$ is skew-symmetric.

Combining the above cases, for any $1\leq i_2,\dots, i_s\leq n-2$, there exists a skew-symmetric matrix $\mathcal{D}$ 
of size $(s-1)\times (s-1)$ such that
\begin{equation}\label{Z'_2}
Z_2^t=
\left[\begin{array}{c}0\cdots 0\\ \mathcal{D}\end{array}\right]\cdot
\left[\begin{array}{ccc|ccc}
	a_{n-2,2}&\cdots & a_{m,2} & -a_{m-1,2}&\cdots & -a_{1,2}\\ 
	\vdots &\ddots & \vdots &\vdots &\ddots &\vdots\\
	a_{n-2,s}&\cdots & a_{m,s} & -a_{m-1,s} & \cdots &-a_{1,s}
\end{array}
\right].
\end{equation}
By abuse of notation, write $ D$ for $ \mathcal{D}$. Our next goal is to show that the second column of $Z_1^t$ depends on the first column of $Z_1^t$ and $M_2$. Recall that 
\[
C=\left[\begin{array}{cc}
\mathbf{a_1}+z_{1,2}\mathbf{a_2}+\cdots +z_{1,s}\mathbf{a_s} & 
\pi \mathbf{a_1}+z_{2,2}\mathbf{a_2}+\cdots +z_{2,s}\mathbf{a_s}
\end{array}\right]
\]
and $B=J_2C^tJ_{n-2}$. By using (\ref{Z'_2}), the $(n-1-i_k)$-th columns ($k=2,\dots, s$) of the relation $B_1+\pi B_2=B_2(M_2Z_2^t)$ give
\begin{equation}\label{Z_1t}
\left[\begin{array}{c}
z_{2,2}\\
\vdots\\
z_{2,s}
\end{array}\right]=\left[\begin{array}{ccc}
d_{2,2}&\cdots & d_{2,s}\\
\vdots &\ddots & \vdots\\
d_{s,2}&\cdots & d_{s,s}
\end{array}\right]
\left[\begin{array}{ccc}
q_{2,1}&\cdots & q_{2,s}\\
\vdots &\ddots & \vdots\\
q_{s,1}&\cdots & q_{s,s}
\end{array}\right]
\left[\begin{array}{c}
1\\
z_{1,2}\\
\vdots\\
z_{1,s}
\end{array}\right]+\pi
\left[\begin{array}{c}
z_{1,2}\\
\vdots\\
z_{1,s}
\end{array}\right]
\end{equation}
where $q_{i,j}=\sum_{k=1}^{m-1}(a_{k,i}a_{n-1-k.j}-a_{n-1-k,i}a_{k,j})$. It is easy to see that the rest columns in $B_1+\pi B_2= B_2(M_2Z_2^t)$ are automatically satisfied when (\ref{Z_1t}) holds.

Finally, consider $CM_1+(M_2Z_2^t)M_2=2\pi M_2$. By using the relations (\ref{Z'_2}) and (\ref{Z_1t})
, it is equivalent to
\[
D\cdot Q=-2\pi I_{s-1},
\]
where
\[
D=\left[\begin{array}{ccc}
d_{2,2}&\cdots & d_{2,s}\\
\vdots &\ddots &\vdots\\
d_{s,2}&\cdots & d_{s,s}
\end{array}\right],\quad
Q=\left[\begin{array}{ccc}
q_{2,2}&\cdots & q_{2,s}\\
\vdots &\ddots &\vdots\\
q_{s,2}&\cdots & q_{s,s}
\end{array}\right].
\]
There are two relations we still need to consider: $B_1M_2=\pi B_2M_2$, and $(M_2Z_2^t)^2=2\pi (M_2Z_2^t)$. We leave it to the reader to verify that these conditions are already satisfied when $D\cdot Q=-2\pi I_{s-1}$.

From all the above, we deduce that an affine neighborhood around $(\calF_0,\Lambda')$ is given by $
_2\calU_{r,s}\cong \Spec O_F[x,z_{1,2},\dots,z_{1,s}, D, M_2]/I$, where
\[
I=(M_2'-[\mathbf{0}\mid I_{s-1}],\ D+D^t,\ D\cdot Q+2\pi I_{s-1}).
\]
Set $\mathbb{A}_{O_F}^{s-1}=\Spec O_F[x,z_{1,2},\dots,z_{1,s}]$. For  $s\geq 2$, we have
\[
_2\U_{r,s}\cong \mathbb{A}_{O_F}^{s-1}\times \Spec O_F [D,M_2]/I.
\]
\end{proof}

In particular, for the signature $(r,s)=(n-3,3)$ we have
\[
D=\left[\begin{array}{cc}
 0 & d\\
-d& 0
\end{array}\right],\quad 
Q=\left[\begin{array}{cc}
 0 & q\\
-q& 0
\end{array}\right],
\]
where we set $d=d_{2,3}=-d_{3,2}, q=q_{2,3}=-q_{3,2}$. Recall that $q=\sum_{k=1}^{m-1}(a_{k,2}a_{n-1-k,3}-a_{n-1-k,2}a_{k,3})$. The equation $D\cdot Q=-2\pi I_{s-1}$ amounts to $d\cdot q-2\pi=0$.

\begin{Corollary}\label{Cor48}
	When $(r,s)=(n-3,3)$, there is an affine chart $_2\U_{n-3,3}$ which is isomorphic to 
 \[V(I)=\Spec O_F[x,z_{1,2}, z_{1,3}, d, (a_{k,1})_{1\leq k\leq n-2}, (a_{k,2})_{1\leq k\leq n-2}, (a_{k,3})_{1\leq k\leq n-2}]/I,
 \]
where
\[
I=(a_{i_2,1}, a_{i_3,1}, a_{i_2,2}-1, a_{i_3,2}, a_{i_2,3}, a_{i_3,3}-1, d\cdot \sum_{k=1}^{m-1}(a_{k,2}a_{n-1-k,3}-a_{n-1-k,2}a_{k,3})-2\pi)
\]
for some $ 1 \leq i_2,i_3 \leq n-2$.
\end{Corollary}
\begin{Remark}\label{notiso2}{\rm
As in Remark \ref{notiso1}, for different choices of $M'_2$ we get different equations from $D\cdot Q = -2 \pi I_s$ and so different affine charts $_2\calU_{r,s}$ which in general are not isomorphic. For instance, in Corollary \ref{Cor48}, there are two different affine charts up to isomorphism when we choose different positions for $i_2, i_3$ (see \S \ref{Odd_s} for details). In Proposition \ref{Prop s=3}, we show that for any such a choice the affine scheme $V (I)$ has
semi-stable reduction over $O_F$.
}\end{Remark}
\section{Splitting Models $\rmM^{\rm spl}$}\label{spl}

In this section, we will first show that the naive splitting model $\calM$ is not flat over $\Spec O_F$ for a general signature $(r,s)$. Then we will define the closed subscheme ${\rm M}^{\rm spl} \subset \calM $. We show that ${\rm M}^{\rm spl}$ is smooth for the signature $(r,s)=(n-1,1)$, and has semi-stable reduction for the signature $(r,s)=(n-2,2)$ and $(r,s)=(n-3,3)$.

\subsection{Naive Splitting Models $\calM$}\label{Naive.spl} By using the explicit description of the two affine charts $_1\U_{r,s}$, $_2\U_{r,s}\subset \calM$, that we obtained in \S \ref{Two affine charts}, we will show that $\calM$ is not flat over $\Spec O_F$.

\begin{Proposition}\label{naive.notflat}
For any signature $(r,s)$, the naive splitting model $\calM$ is not flat over $\Spec O_F$.
\end{Proposition} 
\begin{proof}
It's enough to show: a) When $s$ is odd, the generic fiber $_1\U_{r,s}\otimes_{O_F}F$ is empty and so the point $(\calF_0, t\Lambda_m)$ does not lift to characteristic zero. b) When $s$ is even, the generic fiber $_2\U_{r,s}\otimes_{O_F}F$ is empty and so the point $(\calF_0, \Lambda')$ does not lift to characteristic zero.

Observe that for a skew-symmetric matrix $P$ of size $s\times s$ the determinant of $P$ is equal to $0$ if $s$ is odd. 
Recall that we denote by $_1\U_{r,s}$ the affine chart around the point $(\calF_0, t\Lambda_m)\in \calM$. By Theorem \ref{Simpl.1}, the affine chart $_1\U_{r,s}$ is isomorphic to $\Spec O_F[D,Y]/I$, where
\[
I=(Y'-I_s, \ D+D^t, \ D\cdot Q+2\pi I_s),
\]
for any signature $(r,s)$. Here $D$ is a skew-symmetric matrix of size $s\times s$. Thus, when $s$ is odd, we have $\det (D)=0$. In this case, we can see that the generic fiber $_1\U_{r,s}\otimes_{O_F}F$ is empty since
\[
0=\det(D)\det(Q)=\det(D\cdot Q)=\det(-2\pi I_s)=(-2\pi)^s.
\] 
It follows that the point $(\calF_0, t\Lambda_m)$ does not lift to characteristic zero when $s$ is odd.

Recall that we denote by $_2\U_{r,s}$ the affine chart around the point $(\calF_0, \Lambda')\in \calM$. By Theorem \ref{Simpl.2}, the affine chart $_2\U_{r,s}$ is isomorphic to $\mathbb{A}^{s-1}_{O_F}\times \Spec O_F[D,M_2]/I$, where
\[
I=(M_2'-[0\mid I_s], \ D+D^t, \ D\cdot Q+2\pi I_{s-1}).
\]
When $s$ is even, we see that the generic fiber $_2\U_{r,s}\otimes_{O_F}F$ is empty since $D$ is of size $(s-1)\times (s-1)$ and $0=\det(D\cdot Q)=(-2\pi)^{s-1}$. Therefore, the point $(\calF_0, \Lambda')$ does not lift to characteristic zero when $s$ is even.
\end{proof}


Recall that there is a forgetful morphism $\tau: \calM\rightarrow {\rm M}^\wedge\otimes_O O_F$. From \cite[Remark 5.3]{PR}, \cite[Remark 2.6.10]{PRS}, the authors show that the wedge local model ${\rm M}^\wedge$ is not topological flat for any signature when $n$ is even. Consider the local model ${\rm M}^{\rm loc}$ to be the scheme-theoretic closure in ${\rm M}^\wedge$ of its generic fiber. We have a closed morphism ${\rm M}^{\rm loc}\hookrightarrow {\rm M}^\wedge$. Define the splitting model as the inverse image of ${\rm M}^{\rm loc}$, i.e.
\[
{\rm M}^{\rm spl}:=\tau^{-1}({\rm M}^{\rm loc})=\calM\times_{{\rm M}^\wedge} {\rm M}^{\rm loc}.
\]
The restriction of the forgetful morphism $\tau$ on ${\rm M}^{\rm spl}$ gives 
\[
\tau: {\rm M}^{\rm spl}\rightarrow {\rm M}^{\rm loc}.
\] 


\subsection{The Case $(r,s)=(n-2,2)$}\label{Evensemistable}
By using the explicit description of $ _1\mathcal{U}_{n-2,2}$ in \S \ref{AffineChart1}, we show

\begin{Proposition}\label{SemiStabCase1}
When $(r,s)=(n-2,2)$, $_1\U_{n-2,2}$ has semi-stable reduction over $O_F$. In particular, $_1\U_{n-2,2}$ is regular and has special fiber a reduced divisor with two smooth irreducible components intersecting transversely.
\end{Proposition}

\begin{proof}
From Corollary \ref{CorChart1}, it's enough to show that the scheme $V(I) $ has semi-stable reduction over $O_F$. It suffices to prove that $V(I) $ is regular and its special fiber is reduced with smooth irreducible components that have smooth intersections with correct dimensions (see \cite[\S 2.16]{dJ}). 

First, we rename our variables as follows: $x_i := y_{i,1}$ and $y_i := y_{i,2}$ for $1 \leq i \leq n.$ Thus, $V(I) = \Spec O_F[d,\und x, \und y]/ I$ where 
\[
I= ( x_{i_0}-1,\, y_{j_0} -1 ,\,x_{j_0},\,y_{i_0},\,  d(\sum^m_{j=1}y_{n+1-j}x_{j} -  \sum^n_{j=m+1}y_{n+1-j}x_{j}) -2 \pi).
\]
By symmetry, as above, there two cases to consider: Case 1 when  $1 \leq i_0,\,j_0 \leq m$ 
and Case 2 when  $ 1 \leq i_0 \leq m$ and $ m+1 \leq j_0 \leq n $. 
We consider these two cases separately in this proof.
 
In Case 1 we can assume, for simplicity, that $i_0=1$ and $ j_0=2$. Thus, $V(I)$ is isomorphic to $V(\mathcal{I}) = \Spec R/\mathcal{I}$ where 
\[
 R = O_F[d,(x_i)_{3\leq i \leq n}, (y_i)_{3\leq i \leq n}],
 \]
and  
\[
\mathcal{I} = (d(x_{n-1} +\sum^{n-3}_{i=m+1}x_iy_{n+1-i}- y_n - \sum^m_{i=3}x_iy_{n+1-i}) -2\pi ).
\]
Over the special fiber ($\pi =0$), we have $ V(\mathcal{I} _s) \cong \Spec \bar{R}/\mathcal{I}_s $ where  
\[
\bar{R} = k[d,(x_i)_{3\leq i \leq n}, (y_i)_{3\leq i \leq n}], \quad \mathcal{I}_s= (d(x_{n-1} +\sum^{n-3}_{i=m+1}x_iy_{n+1-i}- y_n - \sum^m_{i=3}x_iy_{n+1-i}) ).
\]
It has two smooth irreducible components $V(\mathcal{I}_i) = \Spec \bar{R}/\mathcal{I}_i$ where 
\[
\mathcal{I}_1 = (d), \quad \mathcal{I}_2 = (x_{n-1} +\sum^{n-3}_{i=m+1}x_iy_{n+1-i}- y_n - \sum^m_{i=3}x_iy_{n+1-i}),
\]
that are isomorphic to $\mathbb{A}_k^{2n-4}$. Their intersection is also smooth and of  dimension $2n-5$. Next, we prove that $V(\mathcal{I} _s)$ is reduced by showing that $\mathcal{I}_s =\mathcal{I}_1 \cap \mathcal{I}_2.$ Clearly, $\mathcal{I} _s  \subset \mathcal{I}_1 \cap \mathcal{I}_2$. Let $g \in \mathcal{I}_1 \cap \mathcal{I}_2 .$ Thus, $g \in \mathcal{I}_1$ and $g = f_1 d$ for $f_1 \in \bar{R}$. Also, $g \in \mathcal{I}_2$ and so $f_1 d \in \mathcal{I}_2$. $\mathcal{I}_2$ is a prime ideal and $d \notin \mathcal{I}_2$. Thus, $f_1 \in \mathcal{I}_2$ and 
\[
g = df_1 \equiv d (x_{n-1} +\sum^{n-3}_{i=m+1}x_iy_{n+1-i}- y_n - \sum^m_{i=3}x_iy_{n+1-i})  f_2 \equiv 0 \, \text{mod} \, \mathcal{I}_s
\]
for $f_2 \in \bar{R}$. Thus, $g \in \mathcal{I}_s$ and so $\mathcal{I}_s = \mathcal{I}_1 \cap \mathcal{I}_2 .$ Lastly, we can easily see that the ideals $ \mathcal{I}_1,\mathcal{I}_2$ are principal over $V(\mathcal{I})$. From the above we deduce that $V(\mathcal{I}) $ is regular (see \cite[Remark 1.1.1]{Ha}).

In Case 2, we set $\Delta=\{1,\dots,n-2\}/ \{m, m+1\}$ and let $i_0=m, j_0=m+1$. Then $V(I)$ is isomorphic to $V(\mathcal{J}) = \Spec R'/\mathcal{J}$ where $R' = O_F[d,(x_i)_{i \in \Delta}, (y_i)_{i \in \Delta}]$ and  
\[
\mathcal{J} = (d(1+ \sum^{m-1}_{i=1}x_iy_{n+1-i}- \sum^n_{i= m+2}x_iy_{n+1-i})-2\pi).
\]
To see that $V(\mathcal{J}) $ has semi-stable reduction over $O_F$ one proceeds exactly as in Case 1. Over the special fiber ($\pi =0$), we have $ V(\mathcal{J} _s) \cong \Spec \bar{R}'/\mathcal{J}_s $ where $ \bar{R} = k[d,(x_i)_{i \in \Delta}, (y_i)_{i \in \Delta}]$ and  
\[
\mathcal{J}_s = (d(1+ \sum^{m-1}_{i=1}x_iy_{n+1-i}- \sum^n_{i= m+2}x_iy_{n+1-i})).
\]
It has two smooth irreducible components $V(\mathcal{J}_i) = \Spec R/\mathcal{J}_i$ where 
\[
\mathcal{J}_1 = (d), \quad \mathcal{J}_2 = (1+ \sum^{m-1}_{i=1}x_iy_{n+1-i}- \sum^n_{i= m+2}x_iy_{n+1-i})
\]
of dimension $2n-4$. (Note that in this case, the second component $V(\mathcal{J}_2)$ is not isomorphic to the affine space $\mathbb{A}_k^{2n-4}$ as in Case 1.) Their intersection is also smooth and of dimension $2n-5$. By using similar arguments as Case 1 we can show that $V(\mathcal{J}_s)$ is reduced and $V(\mathcal{J})$ is regular. 
\end{proof}

Since $_1\U_{n-2,2}$ has semi-stable reduction, and is therefore flat over $O_F$, we have $_1\U_{n-2,2}\subset {\rm M}^{\rm spl}$. Our next goal is to prove that ${\rm M}^{\rm spl}$ has semi-stable reduction over $O_F$ for the signature $(r,s)= (n-2,2)$. From \S \ref{AffineChart1} it's enough to show that $\mathcal{G}$-translates of $_1\U_{n-2,2}$ cover ${\rm M}^{\rm spl}$ (see Theorem \ref{MSemStable}). 

Assume that $s$ is even. Recall from \S \ref{AffineChart1} the $\calG$-equivariant morphism $\tau : {\rm M}^{\rm spl} \rightarrow {\rm M}^{\rm loc}\otimes_{O} O_F$ which is given by $(\mathcal{F}_0,\mathcal{F}_1) \mapsto \mathcal{F}_1$. Let $\tau_s : {\rm M}^{\rm spl}\otimes k \rightarrow {\rm M}^{\rm loc}\otimes k$ be the morphism over the special fiber. Let $(\mathcal{F}_0,\mathcal{F}_1)$ be point of the special fiber of ${\rm M}^{\rm spl}$. Over the special fiber the condition
(4) amounts to $t\mathcal{F}_0 = (0)$ and so $ (0) \subset \mathcal{F}_0 \subset t \Lambda_m \otimes k$ (see \S \ref{LocalModelVariants} for the definition of the naive splitting models).  Hence, we consider the morphism 
\[
\pi : {\rm M}^{\rm spl}\otimes k \rightarrow \Gr(s,n)\otimes k 
\]
given by $ (\mathcal{F}_0,\mathcal{F}_1) \mapsto \mathcal{F}_0.$ Here, $\Gr(s,n)\otimes k$ is the finite dimensional
Grassmannian of $s$-dimensional subspaces of $t \Lambda_m \otimes k.$ This has a section 
\[
\phi : \Gr(s,n)\otimes k \rightarrow {\rm M}^{\rm spl}\otimes k
\]
given by $ \mathcal{F}_0 \mapsto (\mathcal{F}_0,\mathcal{F}_1)$ with $\mathcal{F}_1 = t\Lambda_m \otimes k.$ The image of the section $\phi$ is an irreducible component of $ {\rm M}^{\rm spl}\otimes_{O_F} k$ which is the fiber $\tau^{-1}_s (t\Lambda_m\otimes k)$ over the ``worst point" of ${\rm M}^{\rm loc}\otimes k$, i.e. the unique closed $\calG$-orbit supported in the special fiber (see \cite[\S 5]{PR3}). Hence, $\tau^{-1}_s(t\Lambda_m)$ is isomorphic to the Grassmannian $ \Gr(s,n)\otimes k  $ of dimension $s(n-s)$. 

We equip $t \Lambda_m \otimes k$ with the non-degenerate alternating form $ \langle \, , \, \rangle'$ which is defined as $\langle tv,tw \rangle'  := ( v,tw ) = v^tJ_n w$. Note that this is well-defined (see \cite[\S 5.2]{Richarz} for more details). Denote by $\mathcal{Q}(s,n)$ the closed subscheme of $\phi(\Gr(s,n)\otimes k) $ of dimension $ s(2n-3s+1)/2$ which parametrizes all the isotropic $s$-subspaces $\mathcal{F}_0$ in the $n$-space $t\Lambda_m \otimes k$. 

Next, we give a couple of useful observations concerning the special fiber of the affine chart $_1\U_{r,s}$ (see \S \ref{AffineChart1} for the notation), which will be used in the proof of the following theorem.
 
\begin{Remark}\label{F0_In_U}
{\rm Recall from Theorem \ref{Simpl.1} that over the special fiber of $_1\U_{r,s}$ we have $D=-D^t$, $D\cdot Q=0$ and the minor relations from (\ref{Minors1}). 

a) Observe that if $D =0$ then $Z=0$ which gives $A=0$ since $A = Y Z^t$ over the special fiber. If $A=0$ then $\mathcal{F}_1 = t \Lambda_m$. Hence, $_1\U_{r,s} \cap \tau^{-1}(t\Lambda_m)$ is a smooth irreducible component of dimension $s(n-s)$.

b) Recall that 
\[
\mathcal{F}_0= \left[\begin{array}{c} 
         \pi Y \\        
          Y  
         \end{array}\right] 
         \quad \text{and} \quad Y =  \left[\begin{array}{ccc} 
        \mathbf{y}_1 \ \cdots
         \  \mathbf{y}_s
         \end{array}\right]
\]
where $ \mathbf{y}_i$ are the columns of the matrix $Y.$ By direct calculations we get that the entries $q_{i,j}$ of the matrix $Q$ are given by $q_{i,j} = \langle \mathbf{y}_i,\mathbf{y}_j \rangle'$. Thus, over the special fiber, we have $\langle \mathcal{F}_0,\mathcal{F}_0 \rangle' = Q $ and so from the rank of $ Q$ we read how isotropic $\mathcal{F}_0$ is with respect to $\langle\,,\,\rangle'$. When the rank of the matrix $Q$ is zero, which actually occurs, we have $\langle \mathcal{F}_0,\mathcal{F}_0 \rangle' = 0 $ . 

c) From (a) and (b) we can easily see that $_1\U_{r,s}$ contains points $(\mathcal{F}_0, \mathcal{F}_1)$ where $ \calF_0 \in \mathcal{Q}(s, n)$ and $\calF_1 = t\Lambda_m$.}
\end{Remark}

\begin{Theorem}\label{MSemStable}
 When $(r,s)=(n-2,2)$, $\mathcal{G}$-translates of $_1\U_{n-2,2}$ cover ${\rm M}^{\rm spl}$.
\end{Theorem}

\begin{proof}
From \S \ref{Naive.spl}, we have the forgetful $\calG$-equivariant morphism $\tau : {\rm M}^{\rm spl} \rightarrow {\rm M}^{\rm loc}\otimes_{O} O_F$ given by $(\mathcal{F}_0,\mathcal{F}_1) \mapsto \mathcal{F}_1$. As in \cite[\S 4]{P} and \cite[\S 2.4.2, 5.5]{PR}, the worst
point of ${\rm M}^{\rm loc}\otimes_{O} O_F$ is given by the $k$-valued point $t\Lambda_m\otimes k$.
From the construction of splitting models and local models, 
in order to show that $\mathcal{G}$-translates of $_1\U_{n-2,2}$ cover ${\rm M}^{\rm spl}$ it is enough to prove that $\mathcal{G}$-translates of $_1\U_{n-2,2}$ cover $\tau^{-1}(t\Lambda_m)$. 

Recall that $ \tau^{-1}_s(t\Lambda_m\otimes k) \cong \Gr(2,n)\otimes k$ and $\mathcal{G}_k$ acts via its action by reduction to $t\Lambda_m\otimes k/t^2\Lambda_m\otimes k$. This action factors through the symplectic group $Sp(n)_k$ of the symplectic form $\langle \,,\,\rangle'$ on $ t \Lambda_m\otimes k $ and gives the map $ \mathcal{G}_k \rightarrow Sp(n)_k$. As in \cite[\S 4]{PR3}, $\mathcal{G}_k$ has $Sp(n)_k$ as its maximal reductive quotient. Therefore, the map $\mathcal{G}_k \rightarrow Sp(n)_k$ is surjective.

Next, the $Sp(n)_k$-action on $\Gr(2, n)$ has two orbits. More precisely, the orbits are 
\[
Q(0) = \{ \mathcal{F}_0 \in \Gr(2, n) \mid \calF	_0~\text{contains no isotropic vectors}\}
\]
and
\[
Q(2) = \{ \mathcal{F}_0 \in \Gr(2, n) \mid \calF	_0~\text{is totally isotropic}\}.
\]
Observe that $Q(2)$ is contained in the (Zariski) closure of $Q(0)$, and so $Q(2) = \mathcal{Q}(2,n)$ is the unique closed orbit (see for example \cite[\S 3.1]{BDE} and \cite{ACGH}). Thus, $\mathcal{Q}(2,n)$ is contained in the closure of each orbit $Q(i)$, $i=0,2$. 

Lastly, from Remark \ref{F0_In_U} we have that $_1\U_{n-2,2}\otimes k$ contains points $(\mathcal{F}_0,t\Lambda_m)$ with $\mathcal{F}_0 \in \mathcal{Q}(2,n) $ and so $_1\U_{n-2,2}$ contains points from all the orbits. Therefore, from all the above we deduce that the $\mathcal{G}$-translates of $_1\U_{n-2,2}$ cover $\tau^{-1}(t\Lambda_m)$. 
\end{proof}

\begin{Corollary}\label{cor.(n-2,2)}
When $(r,s)=(n-2,2)$, the scheme ${\rm M}^{\rm spl}$ has semi-stable reduction. In particular, ${\rm M}^{\rm spl}$ is regular and has special fiber a reduced divisor with two smooth irreducible components intersecting transversely.
\end{Corollary}
\begin{proof}
The proof follows from Proposition \ref{SemiStabCase1} and Theorem \ref{MSemStable}. 
\end{proof}

\begin{Remark}
{\rm
Note that the construction of $\phi: \Gr(s,n)\otimes k\cong \tau^{-1}_s(t\Lambda_m\otimes k)$	and the observations in Remark \ref{F0_In_U} are true for any even $s$. In fact, for any even $s$, if we could show that the affine chart $_1\U_{r,s}$ is flat, then $_1\U_{r,s}\subset {\rm M}^{\rm spl}$. By using a similar argument as above, we then could prove that $\calG$-translates of $_1\U_{r,s}$ cover ${\rm M}^{\rm spl}$.
}\end{Remark}


\subsection{The Case $(r,s)=(n-1,1)$ and $(r,s)=(n-3,3)$}\label{Odd_s}

By using the explicit description of $_2\U_{n-3,3}$ in \ref{2ndChart}, we show

\begin{Proposition}\label{Prop s=3}
	When $(r,s)=(n-3,3)$, the affine chart $_2\calU_{n-3,3}$ has semi-stable reduction over $O_F$. In particular, $_2\calU_{n-3,3}$ is regular and has special fiber a reduced divisor with two smooth irreducible components intersecting transversely.
\end{Proposition}

\begin{proof}

Recall from \S \ref{4.2.2} that the quadratic polynomial $\sum_{k=1}^{m-1}(a_{k,2}a_{n-1-k,3}-a_{n-1-k,2}a_{k,3})$  depends on the positions of $i_2, i_3$. By symmetry, we only need to consider the case when $1\leq i_2, i_3\leq m-1$ and the case when $1\leq i_2\leq m-1, m\leq i_3\leq n-2$.

We consider these two cases separately in this proof. In the first case, for simplicity, we can set $i_2=1, i_3=2$. Thus, $_2\calU_{n-3,3}$ is isomorphic to $V(\calI)=\Spec R/\calI$, where 
	\[
	R=O_F[x,d,z_{1,2},z_{1,3},(a_{i,1})_{3\leq i\leq n-2}, (a_{i,2})_{3\leq i\leq n-2}, (a_{i,3})_{3\leq i\leq n-2}],
	\]
and
\[
\calI=(d(a_{n-2,3}+\sum_{k=3}^{m-1}a_{n-1-k,3} a_{k,2}-\sum_{k=m}^{n-4}a_{n-1-k,3}a_{k,2}-a_{n-3,2})-2\pi).	
\]	
It is easy to see that we have two irreducible components in the special fiber, where
\[
I_1=(d),\quad I_2=(a_{n-2,3}+\sum_{k=3}^{m-1}a_{n-1-k,3} a_{k,2}-\sum_{k=m}^{n-4}a_{n-1-k,3}a_{k,2}-a_{n-3,2}).
\]	
Both irreducible components are isomorphic to $\mathbb{A}_k^{3(n-3)}$, so they are smooth with the correct dimension. Their intersection is isomorphic to $\mathbb{A}_k^{3(n-3)-1}$, which is also smooth. Since $I_1, I_2$ are prime ideals, it is enough to prove that the special fiber $V(\calI_s)$ is reduced by showing $\calI_s=I_1\cap I_2$. By using similar arguments as in Proposition \ref{SemiStabCase1}, we get that $V(\calI_s)$ is reduced and $V(\calI)$ is regular. 

In the second case, we can set $\Delta=\{1,\dots,n-2\}/ \{m-1, m\}$, and let $i_2=m-1, i_3=m$. Then $_2\calU_{n-3,3}$ is isomorphic to $V(\mathcal{J})=\Spec R'/\mathcal{J}$, where
\[
R'=O_F[x,d,z_{1,2},z_{1,3},(a_{i,1})_{i\in\Delta}, (a_{i,2})_{i\in\Delta}, (a_{i,3})_{i\in\Delta}], 
\]
and
\[
\calJ=(d(1+\sum_{k=1}^{m-2}a_{n-1-k,3}a_{k,2}-\sum_{k=m+1}^{n-2}a_{n-1-k,3}a_{k,2})-2\pi).
\]
Again, we have two irreducible components over the special fiber:
\[
\calJ_1=(d),\quad \calJ_2=(1+\sum_{k=1}^{m-2}a_{n-1-k,3}a_{k,2}-\sum_{k=m+1}^{n-2}a_{n-1-k,3}a_{k,2}).
\]
Both of them are smooth and of dimension $3(n-3)$ (Note that in this case, the second component $V(\calJ_2)$ is not isomorphic to $\mathbb{A}_k^{3(n-3)}$).  Their intersection is also smooth and of dimension $3n-10$. By using the same arguments as above, we can show that $V(\calJ_s)$ is reduced and $V (\calJ )$ is regular. Thus, $_2\calU_{n-3,3}$ has semi-stable reduction over $O_F$ .
\end{proof}

From Theorem \ref{Prop.s=1} and Proposition \ref{Prop s=3}, we have that $_2\U_{n-1,1}$ is smooth and $_2\U_{n-3,3}$ has semi-stable reduction. Thus, both $ _2\U_{n-1,1}$ and $_2\U_{n-3,3} $ are $O_F$-flat. As in \S \ref{Evensemistable}, we deduce that these affine charts are open subschemes of the corresponding splitting model ${\rm M}^{\rm spl}$. We will prove that ${\rm M}^{\rm spl}$ is smooth when $(r,s)=(n-1,1)$ and ${\rm M}^{\rm spl}$ has semi-stable reduction when $(r,s)=(n-3,3)$ by showing that $\calG$-translates of 
$_2\U_{n-1,1}$ and $_2\U_{n-3,3}$ cover ${\rm M}^{\rm spl}$.

Recall from \S \ref{LocalModelVariants}, the $\calG$-equivariant morphism $\tau: {\rm M}^{\rm spl}\rightarrow\rmM^{\rm loc} \otimes_{O}O_F$ which is given by $(\calF_0,\calF_1)\mapsto \calF_1$ and let $\tau_s: {\rm M}^{\rm spl}\otimes k\rightarrow\rmM^{\rm loc} \otimes k$ be the morphism over the special fiber. When $s$ is odd, the unique closed $\calG$-orbit of ${\rm M}^{\rm loc} \otimes k$ is the orbit of $\Lambda'\otimes k$ (see \S \ref{LocalModelVariants}).
From the construction of the local model it is enough to show that  $\calG_k$-translates of $_2\calU_{n-1,1}\otimes k$ and $_2\calU_{n-3,3}\otimes k$ cover $\tau_s^{-1}(\Lambda'\otimes k)$.

When $(r,s)=(n-1,1)$, the inverse image $\tau_s^{-1}(\Lambda'\otimes k)$ contains points $(\calF_0,\Lambda'\otimes k)$ satisfying $t\Lambda'\subset\calF_0$. Observe that $t\Lambda'=\tspan_{k}\{-e_1\}$ and $\calF_0$ has rank $1$. Thus, $\tau_s^{-1}(\Lambda'\otimes k)=(\text{span}_{k}\{-e_1\}, \Lambda'\otimes k)$ is just a unique point (Note that it is different from the case when $s$ is even, see \S \ref{Evensemistable}, where we have $t\Lambda_m=0$). From \S \ref{AffineChart2}, we see that $_2\calU_{n-1,1}\otimes k$ contains that point and so $\calG_k$-translates of $_2\calU_{n-1,1}\otimes k$ cover $\tau_s^{-1}(\Lambda'\otimes k)$. Hence, $\calG$-translates of $_2\calU_{n-1,1}$ cover ${\rm M}^{\rm spl}$.

When $s$ is odd and $\geq 3$, we have
\[
\tau_s^{-1}(\Lambda'\otimes k)=\{(\calF_0,\Lambda')\mid \tspan_{k}\{-e_1\}\subset \calF_0\subset \Lambda'\otimes k,\quad t\calF_0=(0)\}.
\]
Set $\calF_0=\tspan_k\{e_1, v_2, \dots, v_s\}$, then $t\calF_0=0$ and $\calF_0\subset \Lambda'\otimes k$ imply that $v_2, \dots, v_s$ are a linear combination of $\{-e_2,\dots, -e_m, \pi e_{m+1},\dots, \pi e_{n-1}\}$. Thus, we consider the morphism:
\[
\pi: \tau_s^{-1}(\Lambda'\otimes k)\rightarrow \Gr(s-1,n-2)\otimes k
\] 
given by $(\calF_0,\Lambda')\mapsto \tspan_k\{v_2, \cdots, v_s\}$. Here $\Gr(s-1,n-2)\otimes k$ is the Grassmannian of $(s-1)$-dimensional subspaces of 
\[
\Lambda'/\{-\pi^{-1}e_1, -e_1\}=\tspan_k\{-e_2,\dots, -e_m, \pi e_{m+1},\dots, \pi e_{n-1}\}.
\]
We also have a section:
\[
\phi: \Gr(s-1,n-2)\otimes k\rightarrow {\rm M}^{\rm spl}\otimes k,
\]
given by $W\rightarrow (\tspan_k(e_1, W),\Lambda')$. Similar to \S \ref{Evensemistable}, it is easy to see that the image of the section $\phi$ is $\tau_s^{-1}(\Lambda'\otimes k)$. Hence $\tau_s^{-1}(\Lambda'\otimes k)$ is isomorphic to $\Gr(s-1,n-2)\otimes k$. As in \S \ref{Evensemistable}, we can equip $\tau_s^{-1}(\Lambda'\otimes k)$ with a non-degenerate alternating form, i.e. $\bb tv,tw\pp':=(tv,w)=v^tJ_{n-2}w$ for $tv, tw\in \tau_s^{-1}(\Lambda'\otimes k)$. Denote by $\mathcal{Q}(s-1,n-2)$ the closed subscheme of $\Gr(s-1,n-2)\otimes k$ which parametrizes all isotropic $(s-1)$-subspaces with respect to $\bb \ , \ \pp'$.

The following observations concern the special fiber of $_2\U_{r,s}$ when $s$ is odd. We keep the same notation as in \S \ref{AffineChart2}.
 
\begin{Remark}\label{OddRemarks}
{\rm
(1) Set $\mathbf{a}_i$ the $i$-th column of the matrix $M_2$; $\mathbf{a}_i$ can be expressed as the element $tv\in \tau_s^{-1}(\Lambda'\otimes k)$, where $tv=[ 0~0~a_{1,i}~\cdots ~a_{n-2,i}\mid ~0~\cdots~ 0]^t$. Thus,
\begin{flalign*}
\bb \mathbf{a}_i,\mathbf{a}_j\pp' &=[a_{1,i}~\cdots ~a_{n-2,i}]J_{n-2}[a_{1,j}~\cdots ~a_{n-2,j}]^t	\\
&=\sum_{k=1}^{m-1}(a_{k,i}a_{n-1-k.j}-a_{n-1-k.i}a_{k,j})\\
&= q_{i,j}.
\end{flalign*}
 
 
  
(2) For the special fiber $\rmM^{\rm spl}\otimes k$, consider $(_2\calU_{r,s}\otimes k) \cap \tau_s^{-1}(\Lambda'\otimes k)$. We get $C=0$ since $A=0$. The first column of $C$ is $\mathbf{a}_1+z_{1,2}\mathbf{a}_2+\cdots +z_{1,s}\mathbf{a}_s$. By setting $i=i_2, \dots, i_s$, this gives us $z_{1,2}=\cdots =z_{1,s}=0$ and so $\mathbf{a}_1=\mathbf{0}$. Thus,
 \[
 \calF_0=\left[\begin{array}{c}
 	M_1\\
 	M_2\\ \hline
 	0\\
 	0
 \end{array}\right],
 \]
 where 
 \[
 M_1=\left[\begin{array}{cccc}
 	 0& 0&\cdots & 0\\
 	1& 0&\cdots &0
 	 \end{array}\right],~
 M_2=\left[\begin{array}{cccc}
	\mathbf{0} & \mathbf{a}_2 &\cdots &\mathbf{a}_s
\end{array}\right]=
\left[\begin{array}{cccc}
 	 0 & a_{1,2}&\cdots & a_{1,s}\\
 	\vdots &\vdots &&\vdots\\
 	0& a_{n-2,2}& \cdots & a_{n-2,s}\\ 
 	 \end{array}\right].
 \]
 By definition, $Q=[q_{i,j}]=0$ is equivalent to $\bb \mathbf{a}_i,\mathbf{a}_j \pp'=0$ for $2\leq i, j\leq s$. Therefore, $\pi(\calF_0)$ is totally isotropic under $\bb ~,~ \pp'$ if and only if $Q=0$ (which actually occurs). Thus, 
 \[
 \phi(\mathcal{Q}(s-1,n-2))\subset (_2\calU_{r,s}\otimes k)\cap \tau_s^{-1}(\Lambda'\otimes k).
 \] 
 }\end{Remark}
 
 \begin{Theorem}\label{thm_cover_odd}\
 	When $(r,s)=(n-3,3)$, $\calG$-translates of $_2\calU_{n-3,3}$ cover $\rmM^{\rm spl}$.
 \end{Theorem}
 
 \begin{proof}
 	It is enough to prove that $\calG$-translates of $_2\calU_{n-3,3}$ cover $\tau^{-1}(\Lambda')$. We use similar arguments as in the proof of Theorem \ref{MSemStable}.
Since $\mathcal{G}_k$ acts on $\tau_s^{-1}(\Lambda'\otimes k)\simeq \Gr(2,n-2)\otimes k$, 
this action factors through the symplectic group $Sp(n-2)_k$ of the symplectic form $\langle \,,\,\rangle'$ on $ \Lambda'/ \{-\pi^{-1}e_1, -e_1\} $, and gives the surjective map $ \mathcal{G}_k \rightarrow Sp(n-2)_k$ where
\[
Sp(n-2)_k=\{g\in SL(n-2)_k\mid g^tJ_{n-2}g=J_{n-2}\}.
\]
By \cite[\S 3]{BDE}, the $Sp(n-2)_k$-action on $\Gr(2,n-2)\otimes k$ has two orbits: $Q(0)$ and $Q(2)$, where 
\[
Q(i)=\{\calF_0\in \Gr(2,n-2)\mid \dim(rad(\calF_0))=i\},
\] 
for $i=0, 2$.  Similar to the proof of Theorem \ref{MSemStable}, we have $Q(2)\subset \overline{Q(0)}$. Thus, $Q(2)$ is the unique closed orbit. Observe that $Q(2)$ contains all totally isotropic subspaces $\pi(\calF_0)$, and so $Q(2)= \mathcal{Q}(2,n-2)$. 

From Remark \ref{OddRemarks} (2), we have that $_2\U_{n-3,3}\otimes k$ contains points $(\mathcal{F}_0,\Lambda'\otimes k)$ with $\pi(\mathcal{F}_0) \in \mathcal{Q}(2,n-2) $. By identifying $\tau_s^{-1}(\Lambda'\otimes k)$ with $\Gr(2,n-2)\otimes k$, we get that $_2\U_{n-3,3}$ contains points from all the orbits. Therefore, from all the above we deduce that the $\mathcal{G}$-translates of $_2\U_{n-3,3}$ cover $\tau^{-1}(\Lambda')$.
 \end{proof}
 
 From Propositions \ref{Prop.s=1} and \ref{Prop s=3} and Theorem \ref{thm_cover_odd} we have:
 
 \begin{Corollary}\label{thm_smooth_scheme}
 \begin{itemize}
\item[a)] When $(r,s)=(n-1,1)$, $\rmM^{\rm spl}$ is a smooth scheme.
 	
\item[b)] When $(r,s)=(n-3,3)$, $\rmM^{\rm spl}$ has semi-stable reduction over $O_F$. In particular, $\rmM^{\rm spl}$ is regular and has special fiber a reduced divisor with two smooth irreducible components intersecting transversely.
\end{itemize}
 \end{Corollary}
 
 \begin{Remark}
 {\rm
 For $s$ odd and $\geq 5$, note that we have 
 \[
 \pi: \tau_{s}^{-1}(\Lambda'\otimes k)\rightarroweq \Gr(s-1,n-2)\otimes k,
 \]
 and
 \[
 \mathcal{Q}(s-1,n-2)\subset (_2\U_{r,s}\otimes k)\cap \tau^{-1}_s(\Lambda\otimes k).
 \]
If we can show that the affine chart $_2\U_{r,s}$ is flat, then we have $_2\U_{r,s}\subset {\rm M}^{\rm spl}$. By using a similar argument as the proof of Theorem \ref{thm_cover_odd}, we could prove that $\calG$-translates of $_2\U_{r,s}$ cover ${\rm M}^{\rm spl}$.

 }\end{Remark}
 
\begin{Remark}\label{loc_splt}
    {\rm 
    When $ (r,s) = (n-1,1)$, the local model $ \Mloc \subset {\rm M}^{\wedge}$ is smooth (see \cite[\S 5.3]{PR}). 
    Consider the morphism $\tau :   \rmM^{\rm spl} \longrightarrow  {\rm M}^{\rm loc} \otimes_{O} O_F$. The generic fiber of all of these is the same (see \S \ref{LocalModelVariants}). In particular, from all the above discussion in this section we can see that $  \rmM^{\rm spl} \cong \Mloc \otimes_{O} O_{F}$.
    }
\end{Remark}
 
\section{Application to Shimura Varieties}\label{Shimura} 
In this section, we give the most immediate application of Corollaries \ref{cor.(n-2,2)} and \ref{thm_smooth_scheme} to unitary Shimura varieties. Indeed, by using work of Rapoport and Zink \cite{RZbook} we first construct $p$-adic integral models for the signatures $ (n-s,s)$ with $ s\leq 3$ where the level subgroup at $p$ is the special maximal parahoric subgroup $P_{\{m\}}$ defined in \S \ref{ParahoricSbgrs}. These models have simple and explicit moduli descriptions and are \'etale locally around each point isomorphic to naive local models. Then by using the above Corollaries and producing a linear modification we obtain smooth integral models of the corresponding Shimura varieties  when $s=1$ and semi-stable models when $s=2$ or $s=3$, i.e. they are regular and the irreducible components of the special fiber are smooth divisors crossing normally. (See Theorem \ref{RegLM}.) 

We start with an imaginary quadratic field $K$ and we fix an embedding $\varepsilon: K\rightarrow \mathbb{C}$. Let $W$ be a $n$-dimensional $K$-vector space, equipped with a non-degenerate hermitian form $\phi$. Consider the group $G = \GU_n$ of unitary similitudes for $(W,\phi)$ of dimension $n\geq 3$ over $K$. Assume that $n=2m$ is even. We fix a conjugacy class of homomorphisms $h: \Res_{\mathbb{C}/\mathbb{R}}\mathbb{G}_{m,\mathbb{C}}\rightarrow \GU_n$ corresponding to a Shimura datum $(G,X) = (\GU_n,X_h)$ of signature $(r,s)$. The pair $(G,X)$ gives rise to a Shimura variety $\Sh(G,X)$ over the reflex field $E$. (See \cite[\S 1.1]{PR} and \cite[\S 3]{P} for more details on the description of the unitary Shimura varieties.) Let $p$ be an odd prime number which ramifies in $K$. Set $K_1=K\otimes_{\QQ} \QQ_p$ with a uniformizer $\pi$, and $V=W\otimes_{\QQ} \QQ_p$. 
We assume that the hermitian form $\phi$ is split on $V$, i.e. there is a basis $e_1, \dots, e_n$ such that $\phi(e_i, e_{n+1-j})=\delta_{ij}$ for $1\leq i, j\leq n$. We denote by 
\[
\Lambda_i = \text{span}_{O_{K_1}} \{\pi^{-1}e_1, \dots, \pi^{-1}e_i, e_{i+1}, \dots, e_n\}
\]
the standard lattices in $V$. We can complete this into a self dual lattice chain by setting $\Lambda_{i+kn}:=\pi^{-k}\Lambda_i$ (see \S \ref{Prel.1}). Denote by $P_{\{m\}}$ the stabilizer of $\Lambda_m$ in $G(\mathbb{Q}_p)$. We let $\mathcal{L}$ be the
self-dual multichain consisting of lattices $\{\Lambda_j\}_{j\in n\mathbb{Z} \pm m}$. Here $\mathcal{G} = \underline{{\rm Aut}}(\mathcal{L})$ is the (smooth) group scheme over $\mathbb{Z}_p$ with $P_{\{m\}} = \mathcal{G}(\mathbb{Z}_p)$ the subgroup of $G(\mathbb{Q}_p)$ fixing the lattice chain $\mathcal{L}$.   

Choose also a sufficiently small compact open subgroup $K^p$ of the prime-to-$p$ finite adelic points $G({\mathbb A}_{f}^p)$ of $G$ and set $\mathbf{K}=K^pP_{\{m\}}$. The Shimura variety  ${\rm Sh}_{\mathbf{K}}(G, X)$ with complex points
 \[
 {\rm Sh}_{\mathbf{K}}(G, X)(\mathbb{C})=G(\mathbb{Q})\backslash X\times G({\mathbb A}_{f})/\mathbf{K}
 \]
is of PEL type and has a canonical model over the reflex field $E$. We set $\mathcal{O} = O_{E_v}$ where $v$ the unique prime ideal of $E$ above $(p)$.

We consider the moduli functor $\mathcal{A}^{\rm naive}_{\mathbf{K}}$ over $\Spec \mathcal{O} $ given in \cite[Definition 6.9]{RZbook}:\\
A point of $\mathcal{A}^{\rm naive}_{\mathbf{K}}$ with values in the 
$\mathcal{O}  $-scheme $S$ is the isomorphism class of the following set of data $(A,\iota,\bar{\lambda}, \bar{\eta})$:
\begin{enumerate}
\item An object $(A,\iota)$, where $A$ is an abelian scheme with relative dimension $n$ over $S$ (terminology
of \cite{RZbook}), compatibly endowed with an
action of $\calO$: 
\[ \iota: \calO \rightarrow \text{End} \,A \otimes \mathbb{Z}_p.\]
    \item A $\mathbb{Q}$-homogeneous principal polarization $\bar{\lambda}$ of the $\mathcal{L}$-set $A$.
    \item A $K^p$-level structure
    \[
\bar{\eta} : H_1 (A, {\mathbb A}_{f}^p) \simeq W \otimes  {\mathbb A}_{f}^p \, \text{ mod} \, K^p
    \]
which respects the bilinear forms on both sides up to a constant in $({\mathbb A}_{f}^p)^{\times}$ (see loc. cit. for
details).

The set $A$ should satisfy the determinant condition (i) of loc. cit.
\end{enumerate}

For the definitions of the terms employed here we refer to loc.cit., 6.3–6.8 and \cite[\S 3]{P}. The functor $\mathcal{A}^{\rm naive}_{\mathbf{K}}$ is representable by a quasi-projective scheme over $\mathcal{O}$. Since the Hasse principle is satisfied for the unitary group, we can see as in loc. cit. that there is a natural isomorphism
\[
\mathcal{A}^{\rm naive}_{\mathbf{K}} \otimes_{\calO} E_v = {\rm Sh}_{\mathbf{K}}(G, X)\otimes_{E} E_v.
\]

As is explained in \cite{RZbook} and \cite{P} the naive local model ${\rm M}^{\rm naive}$ is connected to the moduli scheme $\mathcal{A}^{\rm naive}_{\mathbf{K}}$ via the local model diagram 
\[
\mathcal{A}^{\rm naive}_{\mathbf{K}} \ \xleftarrow{\psi_1} \Tilde{\mathcal{A}}^{\rm naive}_{\mathbf{K}} (G,X) \xrightarrow{\psi_2} {\rm M}^{\rm naive}
\]
where the morphism $\psi_1$ is a $\mathcal{G}$-torsor and $\psi_2$ is a smooth and $\mathcal{G}$-equivariant morphism. Therefore, there is a relatively representable smooth
morphism
 \[
 \mathcal{A}^{\rm naive}_{\mathbf{K}} \to [\mathcal{G} \backslash  {\rm M}^{\rm naive}]
 \]
where the target is the quotient algebraic stack.


Next, denote by $ \mathcal{A}^{\rm flat}_{\mathbf{K}} $ the flat closure of ${\rm Sh}_{\mathbf{K}}(G, X)\otimes_{E} E_v$ in $ \mathcal{A}^{\rm naive}_{\mathbf{K}}$. Recall from \S \ref{LocalModelVariants} that the flat closure of $ {\rm M}^{\rm naive} \otimes_{\mathcal{O}} E_v$ in ${\rm M}^{\rm naive}$ is by definition the local model $\Mloc $. By the above we can see, as in \cite{PR}, that there is a relatively representable smooth morphism of relative dimension ${\rm dim} (G)$,
\[\mathcal{A}^{\rm flat}_{\mathbf{K}} \to [\mathcal{G} \backslash \Mloc].\]
This of course implies that $\mathcal{A}^{\rm flat}_{\mathbf{K}}$ is \'etale locally isomorphic to the local model $\Mloc$.

One can now consider a variation of the moduli of abelian schemes $\mathcal{A}_{\mathbf{K}}$ where we add in the moduli problem an additional subspace in the Hodge filtration $ {\rm Fil}^0 (A) \subset H_{dR}^1(A)$ of the universal abelian variety $A$ (see \cite[\S 6.3]{H} for more details) with certain conditions to imitate the definition of the naive splitting model $\mathcal{M}$. (See \S \ref{LocalModelVariants} for the definition of naive splitting models.)
More precisely, $\mathcal{A}_{\mathbf{K}}$ associates to an $O_{K_1}$-scheme $S$ the set of isomorphism classes of objects $(A,\iota,\bar{\lambda}, \bar{\eta},\mathscr{F}_0) $. Here $(A,\iota,\bar{\lambda}, \bar{\eta})$ is an object of  $\mathcal{A}^{\rm naive}_{\mathbf{K}}(S).$ Set $\mathscr{F}_1 := {\rm Fil}^0 (A) $. The final ingredient $\mathscr{F}_0$ of an object of $\mathcal{A}_{\mathbf{K}}$ is the subspace $ \mathscr{F}_0 \subset \mathscr{F}_1 \subset H_{dR}^1(A) $ of rank $s$ which satisfies the following conditions: 
\[
 (\iota(\pi)+\pi ) \mathscr{F}_1 \subset \mathscr{F}_0, \quad (\iota(\pi)-\pi )\mathscr{F}_0 = (0).
\]
There is a forgetful projective morphism
\[
\tau_1 :   \mathcal{A}_{\mathbf{K}} \longrightarrow \mathcal{A}^{\rm naive}_{\mathbf{K}}\otimes_{\mathcal{O}} O_{K_1}
\]
defined by $(A,\iota,\bar{\lambda}, \bar{\eta},\mathscr{F}_0) \mapsto (A,\iota,\bar{\lambda}, \bar{\eta}) $. Moreover, $\mathcal{A}_{\mathbf{K}}$ has the same \'etale local structure as $\mathcal{M}$; it is a ``linear modification" of $\mathcal{A}^{\rm naive}_{\mathbf{K}}\otimes_{\mathcal{O}} O_{K_1}$ in the sense of \cite[\S 2]{P} (see also \cite[\S 15]{PR2}). As we showed in Proposition \ref{naive.notflat}, the scheme $\mathcal{M}$ is never flat and by the above, the
same is true for $\mathcal{A}_{\mathbf{K}}$. 

 \begin{Theorem}\label{RegLM}
Assume that $ (r,s) = (n-1,1)$ or $(n-2,2)$ or $(n-3,3)$. For every $K^p$ as above, there is a
 scheme $\mathcal{A}^{\rm spl}_{\mathbf{K}}$, flat over $\Spec(O_{K_1})$, 
 with
 \[
\mathcal{A}^{\rm spl}_{\mathbf{K}}\otimes_{O_{K_1}} K_1={\rm Sh}_{\mathbf{K}}(G, X)\otimes_{E} K_1,
 \]
 and which supports a local model diagram
  \begin{equation}\label{LMdiagramReg1}
\begin{tikzcd}
&\wti{\mathcal{A}}^{\rm spl}_{\mathbf{K}}(G, X)\arrow[dl, "\pi^{\rm reg}_K"']\arrow[dr, "q^{\rm reg}_K"]  & \\
\mathcal{A}^{\rm spl}_{\mathbf{K}}  &&  {\rm M}^{\rm spl}
\end{tikzcd}
\end{equation}
such that:
\begin{itemize}
\item[a)] $\pi^{\rm reg}_{\mathbf{K}}$ is a $\mathcal{G}$-torsor for the parahoric group scheme $\calG$ that corresponds to $P_{\{m\}}$.

\item[b)] $q^{\rm reg}_{\mathbf{K}}$ is smooth and $\calG$-equivariant.

\item[c)] When $(r,s) = (n-1,1)$, $\mathcal{A}^{\rm spl}_{\mathbf{K}}$ is a smooth scheme. 
\item[c')] When $(r,s) = (n-2,2)$ or $(r,s) =(n-3,3)$, $\mathcal{A}^{\rm spl}_{\mathbf{K}}$ is regular and has special fiber which is a reduced divisor with
normal crossings.
\end{itemize}
 \end{Theorem}
\begin{proof}
From the above we have the local model diagram 
\[
\begin{tikzcd}
&\wti{\mathcal{A}}^{\rm loc}_{\mathbf{K}}(G, X)\arrow[dl, "\psi_1"']\arrow[dr, "\psi_2"]  & \\
\mathcal{A}^{\rm loc}_{\mathbf{K}}  && {\rm M}^{\rm loc}
\end{tikzcd}
\]
where the morphism $\psi_1$ is a $\mathcal{G}$-torsor and $\psi_2$ is a smooth and $\mathcal{G}$-equivariant morphism. Set 
\[
\wti{\mathcal{A}}^{\rm spl}_{\mathbf{K}}(G, X)=\wti{\mathcal{A}}^{\rm loc}_{\mathbf{K}}(G, X)\times_{{\rm M}^{\rm loc}}{\rm M}^{\rm spl}
\]
which carries a diagonal $\mathcal{G}$-action. Since $\tau : {\rm M}^{\rm spl} \rightarrow {\rm M}^{\rm loc}\otimes_{O} O_F$ is projective (see \S \ref{spl}), we can see (\cite[\S 2]{P}) that the quotient
\[
\pi^{\rm reg}_K: \wti{\mathcal{A}}^{\rm spl}_{\mathbf{K}} \longrightarrow \mathcal{A}^{\rm spl}_{\mathbf{K}}:=\calG\backslash \wti{\mathcal{A} }^{\rm spl}_K(G, X)
\]
is represented by a scheme and gives a $\calG$-torsor. (Recall that the morphism $\tau$ induces an isomorphism on the generic fibers.) This is an example of a linear modification, see \cite[\S 2]{P}. The projection gives a smooth $\calG$-morphism
\[
q^{\rm reg}_{\mathbf{K}}: \wti{\mathcal{A}}^{\rm spl}_{\mathbf{K}} \longrightarrow {\rm M}^{\rm spl}
\]
which completes the local model diagram. Property (c) (resp. (c')) follows from 
Theorem \ref{thm_smooth_scheme} (resp. Corollaries \ref{cor.(n-2,2)} and \ref{thm_smooth_scheme}) and properties (a) and (b) which imply that $ \mathcal{A}^{\rm spl}_{\mathbf{K}}$ and
$ {\rm M}^{\rm spl}$ are locally isomorphic for the \'etale topology. 
\end{proof}
\begin{Remarks}
{\rm  Similar results can be obtained for corresponding Rapoport-Zink formal schemes. (See \cite[\S 4]{HPR} for an example of this parallel treatment.)
}
\end{Remarks}  
 
\section{Moduli Description of $\rmM^{\rm spl}$}\label{mod_des}
 
In this section, we show that by adding the ``spin condition" in the naive splitting model $\mathcal{M}$, we get the splitting model ${\rm M}^{\rm spl}$ for signature $s\leq 3$. 

We use the notation of \S \ref{Prelim}, \S \ref{LocalModelVariants} and we set $ W= \wedge^{n}_F (V\otimes_{F_0}F).$ Recall that the symmetric form $( \, , \, )$ splits over $V$ and thus there is a canonical decomposition
\[
W = W_1 \oplus W_{-1}
\]
of $W $ as an $\SO_{2n} \left((\, ,\,)\right)(F)$-representation (see \cite[\S 7]{RSZ}, \cite[\S 2]{Sm3}). For a standard lattice $\Lambda_i$ in $V$ (see \S \ref{Prelim}), we set $W(\Lambda_i)=\wedge^n(\Lambda_i\otimes_{O_{F_0}}O_F )$ and $W(\Lambda_i)_{\pm 1 } = W_{\pm 1} \cap W(\Lambda).$ For an $O_F$-scheme $S$ and $\epsilon \in \{\pm1\}$, we set 
\[
L_{i,\epsilon}(S)=\text{im}[ W(\Lambda_i)_{\epsilon}\otimes_{O_F} \mathcal{O}_S  \longrightarrow  W(\Lambda)\otimes_{O_F} \mathcal{O}_S ].
\]

We now formulate the spin condition: 

(5) (Spin condition) the line bundle $\wedge^n \calF_1 \subset W(\Lambda_m)\otimes_{O_F} \mathcal{O}_S$ is contained in $ L_{m,(-1)^s}(S)$.

The {\it spin splitting model} $\mathcal{M}^{\rm spin}$ is the closed subscheme of $\mathcal{M}$ defined by imposing the spin condition. We have the following inclusions of closed subschemes ${\rm M}^{\rm spl}\subset \calM^{\rm spin}\subset \calM$ which are all equalities between generic fibers (see \cite[\S 7]{PR} and \cite[\S 9]{RSZ}). 
The restriction of $\tau$ on $\calM^{\rm spin}$ gives us
\[
\tau : \mathcal{M}^{\rm spin} \rightarrow {\rm M}^{\rm spin}\otimes_{O} O_F,
\]
where ${\rm M}^{\rm spin}$ is the spin local model defined in \cite{PR}. In particular, ${\rm M}^{\rm spin}$ is the closed subscheme of  ${\rm M}^{\wedge}$ that classifies points given by $\mathcal{F}_1$ which satisfy the spin condition. We want to mention that Smithling \cite[Theorem 1.3]{Sm2} proved that ${\rm M}^{\rm spin}$ is topologically flat. Also, from \cite[Remark 9.9]{RSZ}, we get that the spin condition can be characterized as the following condition:

(5') The rank of $ (t+ \pi )$ on $\mathcal{F}_1$ has the same parity as $s$.\footnote{We interchange $r$ and $s$ in the notation relative to \cite{RSZ}.} 


\begin{Remark}
  {\rm 
Recall from the proof of Proposition \ref{Naive.spl} that the point   $(\calF_0, t\Lambda_m)$ when $s$ is odd and the point $(\calF_0, \Lambda')$ when $s$ is even does not lift to the generic fiber. On the other hand,  we can easily see that these points do not satisfy the condition (5') and so there are not in $\mathcal{M}^{\rm spin}$.
  }
\end{Remark}

\begin{Proposition}\label{Prop 7.2}
\begin{itemize}
\item[a)] When $s$ is even, $_1\U_{r,s}\subset {\calM}^{\rm spin}$.
\item[b)] When $s$ is odd, $_2\U_{r,s}\subset {\calM}^{\rm spin}$.
\end{itemize}
\end{Proposition}
\begin{proof}
It's enough to consider the rank$((t+\pi)\calF_1)$ over the special fiber ($\pi = 0$) for the following two cases.

{\it Case 1:} 
Consider the affine chart $_1\U_{r,s}$ around $(\calF_0, \Lambda_m)$. Recall from \S \ref{AffineChart1} that
\[  
   \mathcal{F}_1 =    \left[\ 
\begin{matrix}[c]
A\\ \hline
I_n  
\end{matrix}\ \right], \quad \mathcal{F}_0 =  \left[\ 
\begin{matrix}[c]
X_1\\ \hline
X_2  
\end{matrix}\ \right].
\]
By $(t + \pi) \mathcal{F}_1 \subset \mathcal{F}_0$, we get
\[
 (t+\pi)\calF_1=
  \left[\ 
\begin{matrix}[c]
\pi_0I_n + \pi A    \\ \hline
A + \pi I_n  
\end{matrix}\ \right] =  \left[\ 
\begin{matrix}[c]
X_1   \\ \hline
X_2
\end{matrix}\ \right]Z^t,
\]
where $Z$ is of size $n\times s$. Since rank$(\calF_0)=s$, the rank of $(t + \pi) \mathcal{F}_1$ amounts to rank$(Z^t)$. Note that $Z^t = D^t\cdot \wti{Y}^t$, where $D$ is of size $s\times s$ and rank$(\wti{Y})=s$. Thus, rank$(Z^t)=$ rank$(D^t)$. Since $D$ is a skew-symmetric matrix, the rank of $D$ is always an even number. Therefore the rank of $(t+\pi)\calF_1$, where $(\calF_0, \calF_1)\in \, _{1}\U_{r,s}$, is always even.

{\it Case 2:} 
Consider the affine chart $_2\U_{r,s}$ around $(\calF_0, \Lambda')$. In this case, from \S \ref{AffineChart2}, we have
\[
\calF_1=\left[\begin{array}{c}
	I_n\\ \hline
	A
\end{array}
\right],\quad
\calF_0=X=\left[\begin{array}{c}
	X_1\\ \hline
	X_2
\end{array}
\right]
\]
and 
\[
(t+\pi)\calF_1=
\left[\begin{array}{c}
	X_1\\ \hline
	X_2
\end{array}
\right]\cdot
\left[\begin{array}{cc}
	Z_1^t & Z_2^t
\end{array}
\right],
\]
where $Z_1$ is of size $2\times s$ and $Z_2$ is of size $(n-2)\times s$. Thus, the rank of $(t + \pi) \mathcal{F}_1$ amounts to rank$([Z_1^t\ Z_2^t])$. 
Over the special fiber we have
\[
Z_1^t=\left[\begin{array}{cc}
	1& 0 \\ 
	z_{1,2} & z_{2,2}\\
	\vdots &\vdots\\
	z_{1,s} & z_{2,s}
	\end{array}
\right], \quad
Z_2^t=\left[\begin{array}{ccc}
	0&\cdots & 0 \\ 
	z_{1,2}'&\cdots & z_{n-2,2}'\\
	\vdots & \ddots & \vdots \\
	z_{1,s}'&\cdots & z_{n-2,s}'
	\end{array}
\right].
\]
We treat $Z_1^t, Z_2^t$ separately. By (\ref{Z'_2}), $Z_2^t$ can be expressed as
\[
Z_2^t=\left[\begin{array}{c}0\cdots 0\\ D\end{array}\right]\cdot \tilde{M}_2=
\left[\begin{array}{ccc}
0&\cdots &0\\
d_{2,2}&\cdots & d_{2,s}\\
\vdots &\ddots &\vdots\\
d_{s,2}&\cdots & d_{s,s}
\end{array}\right]\cdot
\tilde{M}_2,
\]
where $D$ is a skew-symmetric matrix, and $\tilde{M}_2$ contains an $I_{s-1}$-minor. Thus, $Z_2^t$ is generated by the columns of $\left[\begin{array}{c}0\cdots 0\\ D\end{array}\right]$. For $Z_1^t$, we claim that the second column is generated by $\left[\begin{array}{c}0\cdots 0\\ D\end{array}\right]$. From (\ref{Z_1t}), we have
\[
\left[\begin{array}{c}
0\\
z_{2,2}\\
\vdots\\
z_{2,s}
\end{array}\right]=\left[\begin{array}{ccc}
0&\cdots&0\\
d_{2,2}&\cdots & d_{2,s}\\
\vdots &\ddots & \vdots\\
d_{s,2}&\cdots & d_{s,s}
\end{array}\right]
\left[\begin{array}{ccc}
q_{2,1}&\cdots & q_{2,s}\\
\vdots &\ddots & \vdots\\
q_{s,1}&\cdots & q_{s,s}
\end{array}\right]
\left[\begin{array}{c}
1\\
z_{1,2}\\
\vdots\\
z_{1,s}
\end{array}\right]
\]
in the special fiber. Recall that $D\cdot Q=0$, where 
\[
Q=\left[\begin{array}{ccc}
q_{2,2}&\cdots & q_{2,s}\\
\vdots &\ddots &\vdots\\
q_{s,2}&\cdots & q_{s,s}
\end{array}\right].
\]
Then 
\[
\left[\begin{array}{c}
0\\
z_{2,2}\\
\vdots\\
z_{2,s}
\end{array}\right]=\left[\begin{array}{cccc}
0&0&\cdots &0\\
*&0 &\cdots & 0\\
\vdots&\vdots &\ddots & \vdots\\
*& 0&\cdots & 0
\end{array}\right]
\left[\begin{array}{c}
1\\
z_{1,2}\\
\vdots\\
z_{1,s}
\end{array}\right]=
\left[\begin{array}{c}
0\\
*\\
\vdots\\
*
\end{array}\right],
\]
which can be expressed as
\[
\left[\begin{array}{c}
0\\
*\\
\vdots\\
*
\end{array}\right]=q_{2,1}\left[\begin{array}{c}
0\\
d_{2,2}\\
\vdots\\
d_{s,2}
\end{array}\right]+\cdots +q_{s,1}
\left[\begin{array}{c}
0\\
d_{2,s}\\
\vdots\\
d_{s,s}
\end{array}\right].
\]
Therefore, the second column is generated by $\left[\begin{array}{c}0\cdots 0\\ D\end{array}\right]$. From all the above, we deduce that the rank of $[Z_1^t\ Z_2^t]$ amounts to the rank of
\[
\left[\begin{array}{cccc}
	1& 0 &\cdots &0\\ 
	z_{1,2} & d_{2,2}& \cdots & d_{2,s}\\
	\vdots &\vdots& \ddots&\vdots\\
	z_{1,s} & d_{s,2}&\cdots & d_{s,s}
	\end{array}
\right].
\]
Since rank($D$) is even, and the first column is linear independent to the rest columns, the rank of $(t+\pi)\calF_1$, where $(\calF_0, \calF_1)\in _{2}\U_{r,s}$ is always odd. 

Therefore, $_1\U_{r,s}\subset {\calM}^{\rm spin}$ if and only if $s$ is even and $_2\U_{r,s}\subset {\calM}^{\rm spin}$ if and only if $s$ is odd.
\end{proof}

In particular, when $s=2$ we have $_{1}\U_{n-2,2}\subset \calM^{\rm spin}$ and when $s=3$ we have $_{2}\U_{n-3,3}\subset \calM^{\rm spin}$. Moreover, since ${\rm M}^{\rm spin}$ is topologically flat, or in other words, the underlying topological spaces of ${\rm M}^{\rm spin}$ and ${\rm M}^{\rm loc}$ coincide, we can repeat the $\mathcal{G}$-translates argument as in Theorems \ref{MSemStable} and \ref{thm_cover_odd} and obtain

\begin{Proposition}\label{Prop 7.3}
\begin{itemize}
\item[a)] When $(r,s)=(n-2,2)$, $\calG$-translates of $_1\calU_{n-2,2}$ cover $\calM^{\rm spin}$.
\item[b)] When $(r,s)=(n-3,3)$, $\calG$-translates of $_2\calU_{n-3,3}$ cover $\calM^{\rm spin}$.
\end{itemize}	
\end{Proposition}

From Theorems \ref{MSemStable},  \ref{thm_cover_odd}  and Proposition \ref{Prop 7.3}, we deduce that

\begin{Theorem}\label{Thm 7.4}
For the signatures $(n-s,s)$ with $s\leq 3$, we have $\mathcal{M}^{\rm spin}={\rm M}^{\rm spl}$.
\end{Theorem}

Therefore, we obtain a moduli-theoretic description of ${\rm M}^{\rm spl}$ and by Theorem \ref{RegLM} for the corresponding integral model $\mathcal{A}^{\rm spl}_{\mathbf{K}}$. In particular, the closed subscheme $\mathcal{A}^{\rm spl}_{\mathbf{K}} \subset \mathcal{A}_{\mathbf{K}}$ (see \S\ref{Shimura} for the explicit definition of $\mathcal{A}_{\mathbf{K}}$) is obtained by imposing the following additional condition on the moduli problem of $\mathcal{A}_{\mathbf{K}}$:

\medskip

\noindent\emph{Spin condition.} The rank of $(\iota(\pi) + \pi)$ on ${\rm Fil}^0(A)$ has the same parity as~$s$.


\begin{Remark}
{\rm 
Recall the $\calG$-equivariant morphism $\tau: {\rm M}^{\rm spl}\rightarrow {\rm M}^{\rm loc}\otimes_O O_F$, which is given by $(\calF_0, \calF_1)\mapsto \calF_1$. By the above remark we have $\mathcal{M}^{\rm spin}={\rm M}^{\rm spl}$ for the signatures $(n-s,s)$ with $s\leq 3$. Define $\tau_s: {\rm M}^{\rm spl}\otimes k\rightarrow {\rm M}^{\rm loc}\otimes k$ over the special fiber. For a point $ (\mathcal{F}_0, \mathcal{F}_1) \in {\rm M}^{\rm spl}\otimes k$, we have $(0) \subset t \mathcal{F}_1\subset  \mathcal{F}_0$ and $t\mathcal{F}_0 = (0)$ which gives $(0)\subset \mathcal{F}_0 \subset  t\Lambda_m \otimes k$. Also, we have 
\[
\begin{array}{ll}
\mathcal{F}_1 \subset \mathcal{F}_0^{\bot}, &
t\Lambda_m\otimes k \subset \mathcal{F}_0^{\bot},\\
\mathcal{F}_1 \subset t^{-1}(\mathcal{F}_0), &
t\Lambda_m\otimes k \subset t^{-1}(\mathcal{F}_0).
\end{array}
\]
The spaces $t^{-1}(\mathcal{F}_0) , \,\mathcal{F}_0^{\bot}$ have rank $n + s$, $2n-s$ respectively and we indicate with $ \bot$ the orthogonal complement with respect to $(\,,\,)$. From above, we obtain that
\begin{equation}\label{eq 701}
	\calF_1\subset \calF_1+t\Lambda_m\otimes k \subset \mathcal{F}_0^{\bot}\cap t^{-1}(\mathcal{F}_0)\subset t^{-1}(\mathcal{F}_0).
\end{equation}

A) When the signature $(r,s)=(n-2,2)$, we have a morphism $\pi: {\rm M}^{\rm spl}\otimes k\rightarrow \Gr(2,n)\otimes k$ given by $(\calF_0, \calF_1)\mapsto \calF_0$. Let $M_1=\tau_s^{-1}(t\Lambda_m\otimes k)$, and $M_2=\pi^{-1}(\mathcal{Q}(2,n))$. Recall that $M_1$ is isomorphic to the Grassmannian $\Gr(2,n)\otimes k$ of dimension $2n-4$. 
The spin condition translates to: rank($t\calF_1)=$ even. Since $0\subset t\calF_1\subset \calF_0$, we get rank$(t\calF_1)=0$ or $2$, i.e., $t\calF_1=0$ or $t\calF_1=\calF_0$.

Furthermore, from $t\calF_1=0$ we get $\calF_1\subset t\Lambda_m\otimes k$ which implies $\calF_1=t\Lambda_m\otimes k$, so $(\calF_0, \calF_1)\in M_1$. On the other hand, if $t\calF_1=\calF_0$ we have rank$(t(\calF_1+t\Lambda_m\otimes k))=$ rank $(t\calF_1)=2$. Thus, rank$(\calF_1+t\Lambda_m\otimes k)\geq n+2$. By rank$(t^{-1}(\calF_0))=n+2$ and (\ref{eq 701}), we get $\calF_1+t\Lambda_m\otimes k = \mathcal{F}_0^{\bot}\cap t^{-1}(\mathcal{F}_0)= t^{-1}(\mathcal{F}_0)$, which implies $ t^{-1}(\mathcal{F}_0) \subset \mathcal{F}_0^{\bot}$. As in \cite[\S 2.1]{Zac1}, observe that $ \mathcal{F}_0 \in \mathcal{Q}(2,n)$, i.e. $ \langle \mathcal{F}_0,\mathcal{F}_0 \rangle' = 0$, is equivalent to $\text{rank} \, (t^{-1}(\mathcal{F}_0) \cap \mathcal{F}^{\bot}_0) = n+2 $, i.e. $ t^{-1}(\mathcal{F}_0) \subset \mathcal{F}_0^{\bot}$. So $(\calF_0, \calF_1)\in M_2$. 

Therefore, ${\rm M}^{\rm spl}\otimes k  = M_1 \cup M_2$. Moreover, from the above we can easily see that $\tau': M_2\setminus M_1\simeq {\rm M}^{\rm loc} \setminus \{t\Lambda_m\otimes k\}$, so that $M_2$ maps birationally to the special fiber of ${\rm M}^{\rm loc}$. 

Lastly, the fiber $\pi^{-1}(\{\mathcal{F}_0\}) $ with $\mathcal{F}_0 \in \mathcal{Q}(2,n) $ contains all $(\mathcal{F}_0, \mathcal{F}_1) $ with $\mathcal{F}_1$ satisfying $\mathcal{F}_0 \subset  (t^{-1}(\mathcal{F}_0))^{\bot} \subset \mathcal{F}_1 \subset  t^{-1}(\mathcal{F}_0).$ These $ \{\mathcal{F}_1\}$ correspond to isotropic, with respect to $( \,,\,)$, subspaces in the four dimensional space $t^{-1}(\mathcal{F}_0)/ (t^{-1}(\mathcal{F}_0))^{\bot}$ and they are parameterized by the orthogonal Grassmannian $ {\rm OGr}(2,4)$. Therefore, $M_2$ is a ${\rm OGr}(2,4)$-bundle over $ \mathcal{Q}(2,n)$ with dimension $2n-4$ and $M_1, M_2$ intersect transversally over the smooth scheme $\mathcal{Q}(2,n)$ of rank $2n-5$.

B) When the signature $(r,s)=(n-3,3)$, 
the unique closed $\mathcal{G}$-orbit in ${\rm M}^{\rm loc}\otimes k$ is not the point $t\Lambda_m\otimes k$ but the orbit of $\Lambda'\otimes k$ which we denote by $O_{wt}$. Let 
\[
M_1=\tau_s^{-1}(O_{wt})=\{(\calF_0, \calF_1)\in {\rm M}^{\rm spl}\otimes k\mid \calF_1=g\cdot \Lambda' ~\text{for some}~ g\in\calG\}
\]
and
\[
M_2=\{(\calF_0, \calF_1)\in {\rm M}^{\rm spl}\otimes k\mid \calF_0 ~\text{is totally isotropic under}~ \bb \ ,\ \pp'\}.
\]
 The spin condition translates to rank$(t\calF_1)=$ odd. By $0\subset t\calF_1\subset \calF_0$, we get rank$(t\calF_1)=1$ or $3$. 

If rank$(t\calF_1)=3$, we have $t\calF_1=\calF_0$ which implies rank$(\calF_1+t\Lambda_m\otimes k)\geq n+3$. By (\ref{eq 701}), we get $\mathcal{F}_0^{\bot}\cap t^{-1}(\mathcal{F}_0)= t^{-1}(\mathcal{F}_0)$. Thus, $\calF_0$ is totally isotropic under $\bb \ ,\ \pp'$. Next, assume that rank$(t\calF_1)=1$ or in other words consider
\[\xy
	(-12,7)*+{\mathcal{F}_1 \cap t \Lambda_m};
	(-4.5,2.5)*+{\rotatebox{-45}{$\,\, \subset\,\,$}};
	(-4.5,10.5)*+{\rotatebox{45}{$\,\, \subset\,\,$}};
	(1,14)*+{t \Lambda_m};
	(0,0)*+{\mathcal{F}_1 };
	(4,2.5)*+{\rotatebox{45}{$\,\, \subset\,\,$}};
	(6,10.5)*+{\rotatebox{-45}{$\,\, \subset\,\,$}};
	(14,7)*+{ \mathcal{F}_1 + t \Lambda_m };
	\endxy    
	\]
where the quotients arising from all slanted inclusions are finitely generated projective $k$-modules of rank 1. We can easily see that this holds for $ \mathcal{F}_1 = \Lambda'\otimes k$. Since $ t \Lambda_m$ is fixed by the action of $ \calG$ and the $ \calG$-orbit $ O_{wt}$ is closed, we can deduce that rank$(t\calF_1)=1$ if and only if $\mathcal{F}_1 \in O_{wt}$. Therefore, as in (A), we get ${\rm M}^{\rm spl}\otimes k  = M_1 \cup M_2$, where $M_2$ maps birationally to the special fiber of ${\rm M}^{\rm loc}$.

}
\end{Remark}
\smallskip
$\quad$\\
{\bf Acknowledgements:} We thank G. Pappas for several useful comments on a preliminary version of this article. We also thank the referees for their useful suggestions. 

\smallskip
$\quad$\\
{\bf Conflicts of interest:} None. 

\smallskip

$\quad$\\
{\bf Financial support:} The first author was supported by CRC 1442 Geometry: Deformations and Rigidity and the Excellence Cluster of the Universit\"{a}t M\"{u}nster.

\Addresses
\end{document}